\documentclass[reqno]{amsart}
\usepackage{amsfonts,amsmath,amsthm,amssymb,mathrsfs}
\usepackage{color}
\usepackage{amsmath,amssymb,amsfonts,bm}
\usepackage{graphicx}
\usepackage{newunicodechar}
\usepackage{xcolor}
\usepackage{geometry}
\numberwithin{equation}{section}
\usepackage[pdfstartview=FitH,colorlinks=true]{hyperref}
\hypersetup{urlcolor=red, linkcolor=blue,citecolor=red}
\allowdisplaybreaks

\theoremstyle{plain}

\newtheorem{theorem}{Theorem}[section]

\newtheorem{proposition}[theorem]{Proposition}

\newtheorem{lemma}[theorem]{Lemma}
\newtheorem{remark}[theorem]{Remark}

\makeatletter

\begin{document}
\title[KFP equation]
{Analytic smoothing effect for a model of\\ 
Kolmogorov-Prandtl type equations
}

\author[X.-D.Cao \& C.-J. Xu \& Yan Xu]
{Xiao-Dong Cao, Chao-Jiang Xu, Yan Xu}

\address{Xiao-Dong Cao, Chao-Jiang Xu
\newline\indent
School of Mathematics and Key Laboratory of Mathematical MIIT,
\newline\indent
Nanjing University of Aeronautics and Astronautics, Nanjing 210016, China
\newline\indent
Yan Xu
\newline\indent
Department of Mathematical Sciences, Tsinghua University, Beijing 100084, China
}
\email{caoxiaodong@nuaa.edu.cn; xuchaojiang@nuaa.edu.cn; xu-y@mail.tsinghua.edu.cn}

\date{}%\today}

\subjclass[]{35K55, 35Q35}

\keywords{}

\begin{abstract}
In this note, we consider the Cauchy problem for a model of Kolmogorov-Prandtl type equations. For the small perturbation of a constant state, we prove that it admits a unique classical solution with the initial datum in Sobolev space, and it is analytic for any positive times. 
\end{abstract}

\maketitle

\section{Introduction}
The Prandtl equation, known as the foundation of the boundary layer theory introduced by Prandtl in 1904~\cite{Prandtl}, is stated as follows:
\begin{equation}\label{Prandtl equation}
    \begin{cases}
        \partial_T f + f \partial_X f + g \partial_Y f + \partial_X p = \partial_Y^2 f, \ \ T > 0, \ \ X \in \mathbb{R}, \ \ Y > 0, \\
        \partial_X f + \partial_Y g = 0, \\
        \displaystyle f \big|_{Y = 0} = g \big|_{Y = 0} = 0, \quad \lim_{Y \rightarrow +\infty} f = F(T, X), \\
        f\big|_{T = 0} = f_0(X, Y),
    \end{cases}
\end{equation}
here $f(T, X, Y)$ and $g(T, X, Y)$ represent the tangential and normal velocities of the boundary layer, with $Y$ being the scaled normal variable to the boundary, while $F(T, X)$ and $p(T, X)$ are the values on the boundary of the tangential velocity and pressure of the outflow satisfying the Bernoulli law: 
$$
\partial_T F + F \partial_X F + \partial_X p = 0.
$$

Because of the degeneracy in the tangential variable, the mathematical analysis of Prandtl’s boundary layer equations strongly depends on the functional framework. In Sobolev spaces, local well-posedness under the classical monotonicity assumption was first established by Oleinik \cite{OS} using the Crocco transform, and later extended to global-in-time results by Xin–Zhang \cite{XZ-2004} under a favorable pressure condition. More recently, in parallel, Alexandre–Wang–Xu–Yang \cite{AWXY} and Masmoudi–Wong \cite{MW} proved Sobolev well-posedness by energy methods without relying on the Crocco transform. To overcome the derivative loss in the tangential direction, one has to work in more regular classes: local well-posedness was obtained in the analytic setting by Sammartino–Caflisch \cite{SC}, extended to Gevrey-$7/4$ by Gérard-Varet–Masmoudi \cite{G-VM}, and further to Gevrey-2 by Li–Yang \cite{Li-Yang}. 

The distinction between Sobolev and analytic/Gevrey frameworks is decisive for the rigorous justification of Prandtl’s boundary layer theory. Such justification has so far been established only in the analytic framework, initiated by Sammartino–Caflisch \cite{SC,SC2} via the Cauchy–Kovalevskaya approach and later refined by energy methods \cite{202}. Extending this justification to Sobolev initial data still remains an open problem. While Grenier’s result \cite{Grenier} excludes the non-monotone case, the existence theory under the monotonicity assumption \cite{OS,AWXY,MW} suggests that justification in the Sobolev setting may still be possible. A possible pathway was suggested by Li–Wu–Xu \cite{Li-Wu-Xu}, who proved a tangential Gevrey smoothing effect for the Sobolev solutions constructed in \cite{AWXY}. If this smoothing effect could be enhanced to full analyticity, one could then combine it with the analytic justification results \cite{SC,SC2,202} to rigorously validate Prandtl’s ansatz for Sobolev initial data.

However, the Prandtl equation \eqref{Prandtl equation} is a mixed-type degenerate parabolic equation, it presents substantial difficulties for mathematical analysis. Motivated by these challenges, we are led to investigate a model of degenerate nonlinear Kolmogorov-Prandtl type equations, under the monotonicity assumption, by using the Crocco transform, (see \cite{OS,XZ-2004,Z-2024})
\begin{equation}\label{1-1}
\begin{cases}
\partial_t u + y\partial_x u - u^2\,\partial_{y}^2 u = 0, & x, y \in \mathbb{R}, \ \ t > 0, \\
u \big|_{t = 0} = u_0(x, y), &
\end{cases}
\end{equation}
remark that here we omit the boundary problem and the range of the $y-$variable $0 < y < 1$, but consider the Cauchy problem in $\mathbb{R}^2$. Our aim is to prove the local well-posedness with Sobolev initial datum and analytical smoothing effect of Cauchy problem. The study of this relatively simple Prandtl-type model may provide valuable insight and inspiration for the general problem. In this sense, the present paper serves as a first step toward addressing to the Prandtl's ansatz with Sobolev initial datum. 

Let us give the definition of analytic function spaces $\mathcal A(\Omega)$ with $\Omega\subset\mathbb R^{2}$ an open domain. The analytic function spaces $\mathcal A(\Omega)$ is the space of $C^{\infty}$ functions $h$ satisfying:  there exists a constant $C>0$ such that,  for any $\alpha\in\mathbb N^{2}$, 
$$
\left\| \partial^\alpha h \right\|_{L^\infty(\Omega)} \leq C^{\left| \alpha \right| + 1} \alpha !.
$$
Remark that, by using the Sobolev embedding, we can replace the $L^{\infty}$ norm by the $L^{2}$ norm, or the norm in any Sobolev space in the above definition.

The main theorem of this paper is stated as follows: 
\begin{theorem}\label{Theorem}
For any $(u_0-1) \in H^4(\mathbb{R}^2)$,  suppose that $\|u_0-1\|_{H^4(\mathbb{R}^2)}\le \varepsilon_0$ small enough, then there exists $T>0$, such that the Cauchy problem \eqref{1-1} admits a unique solution $u \in C^\infty \left( ]0, T]; \mathcal{A}(\mathbb{R}^2) \right)$. Moreover, for $\eta > 2$, there exists a constant $L > 0$ which depending on $T, \eta, \left\| u_0 -1 \right\|_{H^4(\mathbb{R}^2)}$ such that 
$$
\sup_{0 < t \leq T} t^{(\eta + 1)\ell + \eta n} \left\| \partial_x^\ell \partial_y^n (u(t)-1) \right\|_{H^4(\mathbb{R}^2)} \leq L^{\ell + n + 1} \ell! \ n!, \quad \forall \ell, n \in \mathbb{N}.
$$
\end{theorem}
\begin{remark}
For the local well-posedness, we can take $(u_0-1) \in H^r(\mathbb{R}^2), r>1$ with small norm, and get the solution belongs to  $(u-1) \in L^\infty([0, T];  H^r(\mathbb{R}^2))$. But for the analytic smoothing effect of Cauchy problem, we consider the initial datum belonging to $(u_0-1) \in H^4(\mathbb{R}^2)$, it is also possible to improve. 
\end{remark}

The rest of this paper is organized as follows: to overcome the difficulty caused by the energy estimate of the nonlinear term in the Sobolev space, we first show the uniform bounded of the modified system in Section \ref{section2}. Then by using this uniform estimate, we employ the iteration method in Section \ref{existence} to obtain the existence and first-order energy estimate of the original system. In section \ref{section4}, we establish the energy estimate for high-order derivations. And in Section \ref{section5}, we prove our main regularity result, Theorem \ref{Theorem}.

\section{Uniform bounded of the modified system}\label{section2}
We first change the notation of solution from $u=1+(u-1)=1+v$ to $v$,  then the Cauchy problem \eqref{1-1} transfer to 
\begin{equation}\label{2-1}
\begin{cases}
 \partial_t v + y\partial_x v - (1+v)^2 \partial_{y}^2 v = 0, \ \  (x, y) \in \mathbb{R}^2,\ \   t > 0, \\
v \big|_{t = 0} = v_0(x, y).
\end{cases}
\end{equation}
Take a constant satisfying $0 < \delta \ll 1$. We now consider the following modified system
\begin{equation}\label{equ-m}
\begin{cases}
    \partial_t v_\delta + \langle \delta \tilde y \rangle^{-1}y \partial_x \left(\langle \delta \tilde y \rangle^{-1}v_\delta\right) - \langle \delta \tilde y \rangle^{-1}(1 + g)^2 \partial_y^2 \left(\langle \delta \tilde y \rangle^{-1}v_\delta \right)= 0, \quad \tilde  y=(x, y) \in \mathbb{R}^2, \quad t > 0, \\
    v_\delta \big|_{t = 0} :=\langle \delta \tilde y \rangle^{-1}v_0:=v_{\delta, 0},
\end{cases}
\end{equation}
to establish a uniform bounded estimate (independent of $\delta$) for the approximate solution $v_\delta$ in $H^4(\mathbb{R}^2)$, as a preparation for proving the existence of solution of the Cauchy problem \eqref{2-1}. In this section, we allow all implied constants in the notation $\lesssim$ to be independent of $\delta$.

\begin{proposition}\label{prop2.1-0}
Let $0<\delta\ll1,\, T > 0$, $v_0 \in H^4(\mathbb{R}^2)$ and there exists a sufficiently small constant $\varepsilon_0>0$ such that 
\begin{align}\label{g small} 
     \left\|g\right\|_{L^{\infty}([0, T]; H^{4}(\mathbb R^{2}))}\le\varepsilon_0,\ \   \ \ \ \ 
     \|v_0\|_{H^4(\mathbb R^{2})}\le \varepsilon_0.
\end{align}   
Then the Cauchy problem \eqref{equ-m} admits a unique weak solution. Moreover, we have: 
\begin{align}\label{uniform bounded} 
     \left\|v_\delta(t)\right\|^2_{H^4(\mathbb{R}^2)} + \frac{1}{4}\int_0^t\left\|\partial_y\left(\langle \delta \tilde y \rangle^{-1}v_\delta\right)(s)\right\|^2_{H^4(\mathbb R^2)}ds&\le B_0\left\|v_0\right\|^2_{H^4(\mathbb{R}^2)},\ \ \ 0<t\le T,
\end{align}
where the constant $B_0>0$ is independent of $0 < \delta \ll 1$.
\end{proposition} 

\begin{proof}
\textbf{Existence of weak solution for \eqref{equ-m}:} \ Let 
$$P=-\partial_t+\left(\langle \delta \tilde y \rangle^{-1}y \partial_x(\langle \delta \tilde y \rangle^{-1}\cdot)^*-\langle \delta \tilde y \rangle^{-1}(1 + g)^2 \partial_y^2 \left(\langle \delta \tilde y \rangle^{-1}\cdot \right)\right)^*,$$ 
here the adjoint operator $(\cdot)^*$ is taken with respect to the scalar product in the Hilbert space $H^4(\mathbb R^2)$. For all $\phi\in C^\infty([0, T]; C_c^\infty(\mathbb R^2))$ with $\partial^\alpha\phi(T)=0, \forall \alpha\in\mathbb N^2$, one has
\begin{align*}
\begin{split}
    \left(\phi, P\phi\right)_{H^4(\mathbb R^2)}&=-\frac12\frac{d}{dt}\left\|\phi(t)\right\|^2_{H^4(\mathbb R^2)}+\left(\langle \delta \tilde y \rangle^{-1}y \partial_x(\langle \delta \tilde y \rangle^{-1}\phi), \phi\right)_{H^4(\mathbb R^2)}\\
    &\quad-\left(\langle \delta \tilde y \rangle^{-1}(1 + g)^2 \partial_y^2 \left(\langle \delta \tilde y \rangle^{-1}\phi \right), \phi\right)_{H^4(\mathbb R^2)}\\
    &=-\frac12\frac{d}{dt}\left\|\phi(t)\right\|^2_{H^4(\mathbb R^2)}+I_1+I_2.
\end{split}    
\end{align*}
We now compute both terms separately. From the Leibniz formula, the term $I_1$ can be rewritten as 
\begin{align*}
    I_1&=\sum_{|\alpha|\le4}\left(\partial^{\alpha_1}_x\partial^{\alpha_2}_y\left(\langle \delta \tilde y \rangle^{-1}y \partial_x(\langle \delta \tilde y \rangle^{-1}\phi)\right), \partial^\alpha\phi\right)_{L^2(\mathbb R^2)}\\
    &=\sum_{|\alpha|\le4}\left(\partial^\alpha\left(y \partial_x(\langle \delta \tilde y \rangle^{-1}\phi)\right), \langle \delta \tilde y \rangle^{-1}\partial^\alpha\phi\right)_{L^2(\mathbb R^2)}\\
    &\qquad+\sum_{|\alpha|\le4}\sum_{0<\beta\le\alpha_2}\frac{\alpha_2!}{\beta!(\alpha_2-\beta)!}\left(\partial^{\alpha_1}_x\partial^{\alpha_2-\beta}_y\left(y \partial_x(\langle \delta \tilde y \rangle^{-1}\phi)\right), \partial^\beta\langle \delta \tilde y \rangle^{-1}\partial^\alpha\phi \right)_{L^2(\mathbb R^2)}\\
    &=\sum_{|\alpha|\le4}\left(y \partial_x\partial^\alpha(\langle \delta \tilde y \rangle^{-1}\phi), \langle \delta \tilde y \rangle^{-1}\partial^\alpha\phi\right)_{L^2(\mathbb R^2)}+\sum_{|\alpha|\le4}\left([\partial^\alpha, y]\partial_x(\langle \delta \tilde y \rangle^{-1}\phi), \langle \delta \tilde y \rangle^{-1}\partial^\alpha\phi\right)_{L^2(\mathbb R^2)}\\
    &\qquad+\sum_{|\alpha|\le4}\sum_{0<\beta\le\alpha_2}\frac{\alpha_2!}{\beta!(\alpha_2-\beta)!}\left(\partial^{\alpha_1}_x\partial^{\alpha_2-\beta}_y\left(y \partial_x(\langle \delta \tilde y \rangle^{-1}\phi)\right), \partial^\beta\langle \delta \tilde y \rangle^{-1}\partial^\alpha\phi\right)_{L^2(\mathbb R^2)}=I_{1, 1}+I_{1, 2}+I_{1, 3}.
\end{align*}
For the term $I_{1, 1}$. By using the Leibniz formula, we have 
\begin{align*}
   I_{1, 1}&=\sum_{|\alpha|\le4}\left(y \partial_x\partial^\alpha(\langle \delta \tilde y \rangle^{-1}\phi), \langle \delta \tilde y \rangle^{-1}\partial^\alpha(\langle \delta \tilde y \rangle\langle \delta \tilde y \rangle^{-1}\phi)\right)_{L^2(\mathbb R^2)}\\
   &=\sum_{|\alpha|\le4}\left(y \partial_x\partial^\alpha(\langle \delta \tilde y \rangle^{-1}\phi), \langle \delta \tilde y \rangle^{-1}\partial^{\alpha_2}_y(\langle \delta \tilde y \rangle\partial^{\alpha_1}_x(\langle \delta \tilde y \rangle^{-1}\phi))\right)_{L^2(\mathbb R^2)}\\
   &=\sum_{|\alpha|\le4}\sum_{0<\beta\le\alpha_2}\frac{\alpha_2!}{\beta!(\alpha_2-\beta)!}\left(y \partial_x\partial^\alpha(\langle \delta \tilde y \rangle^{-1}\phi), \langle \delta \tilde y \rangle^{-1}\partial^\beta_y\langle \delta \tilde y \rangle \partial^{\alpha_1}_x\partial^{\alpha_2-\beta}_y(\langle \delta \tilde y \rangle^{-1}\phi)\right)_{L^2(\mathbb R^2)}\\
   &\qquad+\sum_{|\alpha|\le4}\left(y \partial_x\partial^\alpha(\langle \delta \tilde y \rangle^{-1}\phi), \partial^\alpha(\langle \delta \tilde y \rangle^{-1}\phi)\right)_{L^2(\mathbb R^2)}.
\end{align*}
From the integration by parts, we have 
$$\sum_{|\alpha|\le4}\left(y \partial_x\partial^\alpha(\langle \delta \tilde y \rangle^{-1}\phi), \partial^\alpha(\langle \delta \tilde y \rangle^{-1}\phi)\right)_{L^2(\mathbb R^2)}=0,$$
and 
\begin{align*}
   &\sum_{|\alpha|\le4}\sum_{0<\beta\le\alpha_2}\frac{\alpha_2!}{\beta!(\alpha_2-\beta)!}\left(y \partial_x\partial^\alpha(\langle \delta \tilde y \rangle^{-1}\phi), \langle \delta \tilde y \rangle^{-1}\partial^\beta_y\langle \delta \tilde y \rangle \partial^{\alpha_1}_x\partial^{\alpha_2-\beta}_y(\langle \delta \tilde y \rangle^{-1}\phi)\right)_{L^2(\mathbb R^2)}\\
    &=-\sum_{|\alpha|\le4}\sum_{0<\beta\le\alpha_2}\frac{\alpha_2!}{\beta!(\alpha_2-\beta)!}\left(\partial^\alpha(\langle \delta \tilde y \rangle^{-1}\phi), y\langle \delta \tilde y \rangle^{-1}\partial^\beta_y\langle \delta \tilde y \rangle \partial^{\alpha_1+1}_x\partial^{\alpha_2-\beta}_y(\langle \delta \tilde y \rangle^{-1}\phi)\right)_{L^2(\mathbb R^2)}.
\end{align*}
Since $\left|y \langle \delta \tilde y \rangle^{-1}\partial^\beta_y \langle \delta \tilde y \rangle\right|\lesssim 1$ for $\beta\ge1$, then using the Cauchy-Schwarz inequality
\begin{align*}
   |I_{1, 1}|&\le\sum_{|\alpha|\le4}\sum_{0<\beta\le\alpha_2}\frac{\alpha_2!}{\beta!(\alpha_2-\beta)!}\left|\left(\partial^\alpha(\langle \delta \tilde y \rangle^{-1}\phi), y\langle \delta \tilde y \rangle^{-1}\partial^\beta_y\langle \delta \tilde y \rangle \partial^{\alpha_1+1}_x\partial^{\alpha_2-\beta}_y(\langle \delta \tilde y \rangle^{-1}\phi)\right)_{L^2(\mathbb R^2)}\right|\\
    &\le C_1\sum_{|\alpha|\le4}\sum_{0<\beta\le\alpha_2}\frac{\alpha_2!}{\beta!(\alpha_2-\beta)!}\left\|\partial^\alpha(\langle \delta \tilde y \rangle^{-1}\phi)\right\|_{L^2(\mathbb R^2)}\left\| \partial^{\alpha_1+1}_x\partial^{\alpha_2-\beta}_y(\langle \delta \tilde y \rangle^{-1}\phi)\right\|_{L^2(\mathbb R^2)}\\
    &\le\tilde C_1\left\|\langle \delta \tilde y \rangle^{-1}\phi\right\|^2_{H^4(\mathbb R^2)}.
\end{align*}
For the term $I_{1, 2}$. Since 
\begin{align}\label{C-y}
[\partial^\alpha, y]=[\partial^{\alpha_1}_x\partial^{\alpha_2}_y, y]=[\partial^{\alpha_2}_y, y]\partial^{\alpha_1}_x= 
\begin{cases}
	0,&\quad \alpha_2=0,\\
	\alpha_2 \partial^{\alpha_1}_x\partial^{\alpha_2-1}_y, &\quad \alpha_2\ge1,
\end{cases}
\end{align}
we have
\begin{align*}
    I_{1, 2}=\sum_{|\alpha|\le4, \alpha_2\ne0} \alpha_2\left(\partial^{\alpha_1+1}_x\partial^{\alpha_2-1}_y(\langle \delta \tilde y \rangle^{-1}\phi), \langle \delta \tilde y \rangle^{-1}\partial^\alpha\phi\right)_{L^2(\mathbb R^2)}.
\end{align*}
By using the Leibniz formula, we have 
\begin{align*}
   I_{1, 2}&=\sum_{|\alpha|\le4, \alpha_2\ne0} \alpha_2\left(\partial^{\alpha_1+1}_x\partial^{\alpha_2-1}_y(\langle \delta \tilde y \rangle^{-1}\phi), \langle \delta \tilde y \rangle^{-1}\partial^\alpha(\langle \delta \tilde y \rangle\langle \delta \tilde y \rangle^{-1}\phi)\right)_{L^2(\mathbb R^2)}\\
   &=\sum_{|\alpha|\le4, \alpha_2\ne0}\sum_{\beta\le\alpha_2}\frac{\alpha_2 \cdot \alpha_2!}{\beta!(\alpha_2-\beta)!}\left(\partial^{\alpha_1+1}_x\partial^{\alpha_2-1}_y(\langle \delta \tilde y \rangle^{-1}\phi), \langle \delta \tilde y \rangle^{-1}\partial^\beta_y\langle \delta \tilde y \rangle\partial^{\alpha_1}_x\partial^{\alpha_2-\beta}_y(\langle \delta \tilde y \rangle^{-1}\phi)\right)_{L^2(\mathbb R^2)}.
\end{align*}
Since $\left| \langle \delta \tilde y \rangle^{-1}\partial^\beta\langle \delta \tilde y \rangle\right|\lesssim 1$ for $\beta\ge0$, from the Cauchy-Schwarz inequality we have
\begin{align*}
   |I_{1, 2}|& \leq C_2\sum_{|\alpha|\le4, \alpha_2\ne0}\sum_{\beta\le\alpha_2}\frac{\alpha_2 \cdot \alpha_2!}{\beta!(\alpha_2-\beta)!}\left\|\partial^{\alpha_1+1}_x\partial^{\alpha_2-1}_y(\langle \delta \tilde y \rangle^{-1}\phi)\right\|_{L^2(\mathbb R^2)} \left\|\partial^{\alpha_1}_x\partial^{\alpha_2-\beta}_y(\langle \delta \tilde y \rangle^{-1}\phi)\right\|_{L^2(\mathbb R^2)}\\
   &\leq\tilde C_2\left\|\langle \delta \tilde y \rangle^{-1}\phi\right\|^2_{H^4(\mathbb R^2)}.
\end{align*}
For the term $I_{1, 3}$. Since we can deduce the following result by using the Leibniz formula,  
\begin{align*}
    &\left(\partial^{\alpha_1}_x\partial^{\alpha_2-\beta}_y\left(y \partial_x(\langle \delta \tilde y \rangle^{-1}\phi)\right), \partial^\beta\langle \delta \tilde y \rangle^{-1}\partial^\alpha\phi\right)_{L^2(\mathbb R^2)}\\
    &=\left(\partial^{\alpha_1}_x\partial^{\alpha_2-\beta}_y\left(y \partial_x(\langle \delta \tilde y \rangle^{-1}\phi)\right), \partial^\beta\langle \delta \tilde y \rangle^{-1}\partial^\alpha(\langle \delta \tilde y \rangle\langle \delta \tilde y \rangle^{-1}\phi)\right)_{L^2(\mathbb R^2)}\\
    &=\sum_{\sigma\le\alpha_2}\frac{\alpha_2!}{\sigma!(\alpha_2-\sigma)!}\left(\partial^{\alpha_1}_x\partial^{\alpha_2-\beta}_y\left(y \partial_x(\langle \delta \tilde y \rangle^{-1}\phi)\right), \partial^\beta\langle \delta \tilde y \rangle^{-1}\partial^\sigma_y\langle \delta \tilde y \rangle \partial^{\alpha_1}_x\partial^{\alpha_2-\sigma}_y(\langle\delta \tilde y \rangle^{-1}\phi)\right)_{L^2(\mathbb R^2)}\\
    &=\sum_{\sigma\le\alpha_2}\frac{\alpha_2!}{\sigma!(\alpha_2-\sigma)!}\left(\partial^{\alpha_1+1}_x\partial^{\alpha_2-\beta}_y\left(\langle \delta \tilde y \rangle^{-1}\phi\right), y\partial^\beta\langle \delta \tilde y \rangle^{-1}\partial^\sigma_y\langle \delta \tilde y \rangle \partial^{\alpha_1}_x\partial^{\alpha_2-\sigma}_y(\langle\delta \tilde y \rangle^{-1}\phi)\right)_{L^2(\mathbb R^2)}\\
    &\qquad+\sum_{\sigma\le\alpha_2}\frac{\alpha_2!}{\sigma!(\alpha_2-\sigma)!}\left([\partial^{\alpha_1}_x\partial^{\alpha_2-\beta}_y, y]\partial_x(\langle \delta \tilde y \rangle^{-1}\phi), \partial^\beta\langle \delta \tilde y \rangle^{-1}\partial^\sigma_y\langle \delta \tilde y \rangle \partial^{\alpha_1}_x\partial^{\alpha_2-\sigma}_y(\langle\delta \tilde y \rangle^{-1}\phi)\right)_{L^2(\mathbb R^2)}, 
\end{align*}
together with the fact that $\left|y\partial^\beta\langle \delta \tilde y \rangle^{-1}\partial^\sigma_y\langle \delta \tilde y \rangle\right|\lesssim1$ for all $\beta,\sigma\ge1$ and \eqref{C-y}, then the Cauchy-Schwarz inequality gives that 
\begin{align*}
    |I_{1, 3}|&=\sum_{|\alpha|\le4}\sum_{0<\beta\le\alpha_2}\frac{\alpha_2!}{\beta!(\alpha_2-\beta)!}\left|\left(\partial^{\alpha_1}_x\partial^{\alpha_2-\beta}_y\left(y \partial_x(\langle \delta \tilde y \rangle^{-1}\phi)\right), \partial^\beta\langle \delta \tilde y \rangle^{-1}\partial^\alpha\phi\right)_{L^2(\mathbb R^2)}\right|\\ 
    &\le C_3\left\|\langle \delta \tilde y \rangle^{-1}\phi\right\|^2_{H^4(\mathbb R^2)}.
\end{align*}
For $I_2$, by using the Leibniz formula, we can write it as 
\begin{align*}
    I_2&=-\sum_{|\alpha|\le4}\left(\partial^{\alpha_1}_x\partial^{\alpha_2}_y\left(\langle \delta \tilde y \rangle^{-1}(1 + g)^2 \partial_y^2 \left(\langle \delta \tilde y \rangle^{-1}\phi \right)\right), \partial^\alpha\phi\right)_{L^2(\mathbb R^2)}\\
    &=-\sum_{|\alpha|\le4}\left(\langle \delta y \rangle^{-1}\partial^{\alpha}\left((1 + g)^2 \partial_y^2 \left(\langle \delta \tilde y \rangle^{-1}\phi \right)\right), \partial^\alpha\phi\right)_{L^2(\mathbb R^2)}\\
    &\qquad-\sum_{|\alpha|\le4}\sum_{0<\beta\le\alpha_2}\frac{\alpha_2!}{\beta!(\alpha_2-\beta)!}\left(\partial^\beta_y\langle \delta \tilde y \rangle^{-1}\partial^{\alpha_1}_x\partial^{\alpha_2-\beta}_y\left((1 + g)^2 \partial_y^2 \left(\langle \delta \tilde y \rangle^{-1}\phi \right)\right), \partial^\alpha\phi\right)_{L^2(\mathbb R^2)}\\
    &=I_{2, 1}+I_{2, 2}.
\end{align*}
For the term $I_{2, 1}$. By the Leibniz formula, one has
\begin{align*}
    I_{2, 1}&=-\sum_{|\alpha|\le4}\left(\partial^{\alpha}\left((1 + g)^2 \partial_y^2 \left(\langle \delta \tilde y \rangle^{-1}\phi \right)\right), \langle \delta \tilde y \rangle^{-1}\partial^{\alpha_1}_x\partial^{\alpha_2}_y(\langle \delta \tilde y \rangle\langle \delta \tilde y \rangle^{-1}\phi)\right)_{L^2(\mathbb R^2)}\\
    &=-\sum_{|\alpha|\le4}\left(\partial^{\alpha}\left((1 + g)^2 \partial_y^2 \left(\langle \delta \tilde y \rangle^{-1}\phi \right)\right), \partial^{\alpha}(\langle \delta \tilde y \rangle^{-1}\phi)\right)_{L^2(\mathbb R^2)}\\
    &\quad-\sum_{|\alpha|\le4}\sum_{0<\beta\le\alpha_2}\frac{\alpha_2!}{\beta!(\alpha_2-\beta)!}\left(\partial^{\alpha}\left((1 + g)^2 \partial_y^2 \left(\langle \delta \tilde y \rangle^{-1}\phi \right)\right), \langle \delta \tilde y \rangle^{-1}\partial^\beta_y\langle \delta \tilde y \rangle\partial^{\alpha_1}_x\partial^{\alpha_2-\beta}_y(\langle \delta \tilde y \rangle^{-1}\phi)\right)_{L^2(\mathbb R^2)} \\
    &=I_{2, 1, 1} + I_{2, 1, 2}.
\end{align*}
For the term $I_{2, 1, 1}$, we can divide it into 
\begin{align*}
    I_{2, 1, 1}&=-\sum_{|\alpha|\le4}\left(\partial^{\alpha}\partial_y\left((1 + g)^2 \partial_y \left(\langle \delta \tilde y \rangle^{-1}\phi \right)\right), \partial^{\alpha}(\langle \delta \tilde y \rangle^{-1}\phi)\right)_{L^2(\mathbb R^2)}\\
    &+\sum_{|\alpha|\le4}\left(\partial^{\alpha}\left(\partial_y(1 + g)^2 \partial_y \left(\langle \delta \tilde y \rangle^{-1}\phi \right)\right), \partial^{\alpha}(\langle \delta \tilde y \rangle^{-1}\phi)\right)_{L^2(\mathbb R^2)}\\
    &=\sum_{|\alpha|\le4}\left(\partial^{\alpha}\left((1 + g)^2 \partial_y \left(\langle \delta \tilde y \rangle^{-1}\phi \right)\right), \partial^{\alpha}\partial_y(\langle \delta \tilde y \rangle^{-1}\phi)\right)_{L^2(\mathbb R^2)}\\
    &+\sum_{|\alpha|\le4}\left(\partial^{\alpha}\left(\partial_y(1 + g)^2 \partial_y \left(\langle \delta \tilde y \rangle^{-1}\phi \right)\right), \partial^{\alpha}(\langle \delta \tilde y \rangle^{-1}\phi)\right)_{L^2(\mathbb R^2)},
\end{align*}
from Leibniz formula 
\begin{align*}
    &\sum_{|\alpha|\le4}\left(\partial^{\alpha}\left((1 + g)^2 \partial_y \left(\langle \delta \tilde y \rangle^{-1}\phi \right)\right), \partial^{\alpha}\partial_y(\langle \delta \tilde y \rangle^{-1}\phi)\right)_{L^2(\mathbb R^2)}\\
    &=\sum_{|\alpha|\le4}\left((1 + g)^2 \partial^{\alpha}\partial_y \left(\langle \delta \tilde y \rangle^{-1}\phi \right), \partial^{\alpha}\partial_y(\langle \delta \tilde y \rangle^{-1}\phi)\right)_{L^2(\mathbb R^2)}\\
    &\quad+\sum_{|\alpha|\le4}\sum_{0<\beta\le\alpha}\frac{\alpha!}{\beta!(\alpha-\beta)!}\left(\partial^{\beta}(1 + g)^2 \partial^{\alpha-\beta}\partial_y \left(\langle \delta \tilde y \rangle^{-1}\phi \right), \partial^{\alpha}\partial_y(\langle \delta \tilde y \rangle^{-1}\phi)\right)_{L^2(\mathbb R^2)}.
\end{align*}
Using the assumption of  $\|g\|_{L^\infty([0, T]; H^4(\mathbb{R}^2))}\leq \varepsilon_0 $, we have
$$
(1+g)^2 \ge \frac 14.
$$
Then
\begin{align*}
    \sum_{|\alpha|\le4}\left((1 + g)^2 \partial^{\alpha}\partial_y \left(\langle \delta \tilde y \rangle^{-1}\phi \right), \partial^{\alpha}\partial_y(\langle \delta \tilde y \rangle^{-1}\phi)\right)_{L^2(\mathbb R^2)}&\ge\frac14\sum_{|\alpha|\le4}\left( \partial^{\alpha}\partial_y \left(\langle \delta \tilde y \rangle^{-1}\phi \right), \partial^{\alpha}\partial_y(\langle \delta \tilde y \rangle^{-1}\phi)\right)_{L^2(\mathbb R^2)}\\
    &=\frac14\left\|\partial_y\left(\langle \delta \tilde y \rangle^{-1}\phi\right)\right\|^2_{H^4(\mathbb R^2)}.
\end{align*}
The Sobolev embedding $H^{1+}(\mathbb R^2)\subset L^\infty(\mathbb R^2)$ gives that 
\begin{align}\label{ineq1}
    \sum_{|\alpha|\le 2}\|\partial^\alpha (1+g)^2\| _{L^\infty([0, T]; L^\infty(\mathbb{R}^2))}\le C_0\|g\|^2_{L^\infty([0, T]; H^4(\mathbb{R}^2))}\le C_0\epsilon^2_0.
\end{align}
So that, we consider the following two cases separately
\begin{align*}
    &\left|\left(\partial^{\beta}(1 + g)^2 \partial^{\alpha-\beta}\partial_y \left(\langle \delta \tilde y \rangle^{-1}\phi \right), \partial^{\alpha}\partial_y(\langle \delta \tilde y \rangle^{-1}\phi)\right)_{L^2(\mathbb R^2)}\right|\\
    &\le\left\{
\begin{array}{ll}
   \left\|\partial^{\beta}(1 + g)^2 \right\|_{L^\infty(\mathbb R^2)}\left\|\partial^{\alpha-\beta}\partial_y \left(\langle \delta \tilde y \rangle^{-1}\phi\right)\right\|_{L^2(\mathbb R^2)} \left\|\partial^{\alpha}\partial_y(\langle \delta \tilde y \rangle^{-1}\phi)\right\|_{L^2(\mathbb R^2)}, \ 1\le|\beta|\le2,\\
   \\
   \left\|\partial^{\beta}(1 + g)^2 \right\|_{L^2(\mathbb R^2)}\left\|\partial^{\alpha-\beta}\partial_y \left(\langle \delta \tilde y \rangle^{-1}\phi\right)\right\|_{L^\infty(\mathbb R^2)} \left\|\partial^{\alpha}\partial_y(\langle \delta \tilde y \rangle^{-1}\phi)\right\|_{L^2(\mathbb R^2)}, \ 3\le|\beta|\le4,
\end{array}
\right.
\end{align*}
thus we can deduce
\begin{align*}
    &\left|\sum_{|\alpha|\le4}\sum_{0<\beta\le\alpha}\frac{\alpha!}{\beta!(\alpha-\beta)!}\left(\partial^{\beta}(1 + g)^2 \partial^{\alpha-\beta}\partial_y \left(\langle \delta \tilde y \rangle^{-1}\phi \right), \partial^{\alpha}\partial_y(\langle \delta \tilde y \rangle^{-1}\phi)\right)_{L^2(\mathbb R^2)}.\right|\\
    &\le\frac{1}{32}\left\|\partial_y\left(\langle \delta \tilde y \rangle^{-1}\phi\right)\right\|^2_{H^4(\mathbb R^2)}+\tilde C_0\|g\|^4_{L^\infty([0, T]; H^4(\mathbb{R}^2))}\left\|\langle \delta \tilde y \rangle^{-1}\phi\right\|^2_{H^4}.
\end{align*}
We now estimate the second term in $I_{2, 1, 1}$, we first divide it into following two terms
\begin{align*}
    &\sum_{|\alpha|\le4}\left(\partial^{\alpha}\left(\partial_y(1 + g)^2 \partial_y \left(\langle \delta \tilde y \rangle^{-1}\phi \right)\right), \partial^{\alpha}(\langle \delta \tilde y \rangle^{-1}\phi)\right)_{L^2(\mathbb R^2)}\\
    &=\left(\partial_y(1 + g)^2 \partial_y \left(\langle \delta \tilde y \rangle^{-1}\phi \right), \langle \delta \tilde y \rangle^{-1}\phi\right)_{L^2(\mathbb R^2)}\\
    &\quad+\sum_{0<|\alpha|\le4}\left(\partial^{\alpha}\left(\partial_y(1 + g)^2 \partial_y \left(\langle \delta \tilde y \rangle^{-1}\phi \right)\right), \partial^{\alpha}(\langle \delta \tilde y \rangle^{-1}\phi)\right)_{L^2(\mathbb R^2)}.
\end{align*}
Using \eqref{ineq1}, one has
\begin{align*}
    &\left|\left(\partial_y(1 + g)^2 \partial_y \left(\langle \delta \tilde y \rangle^{-1}\phi \right), \langle \delta \tilde y \rangle^{-1}\phi\right)_{L^2(\mathbb R^2)}\right|\\
    &\le\left\|\partial_y(1 + g)^2 \right\|_{L^\infty(\mathbb R^2)}\left\|\partial_y \left(\langle \delta \tilde y \rangle^{-1}\phi \right)\right\|_{L^2(\mathbb R^2)}\left\|\langle \delta \tilde y \rangle^{-1}\phi\right\|_{L^2(\mathbb R^2)}\\
    &\le\left\|g\right\|_{H^4(\mathbb R^2)}\left\|\partial_y \left(\langle \delta \tilde y \rangle^{-1}\phi \right)\right\|_{L^2(\mathbb R^2)}\left\|\langle \delta \tilde y \rangle^{-1}\phi\right\|_{L^2(\mathbb R^2)}.
\end{align*}
For the case of $|\alpha|\ne0$, the calculation directly yields
\begin{align*}
    &\sum_{0<|\alpha|\le4}\left(\partial^{\alpha}\left(\partial_y(1 + g)^2 \partial_y \left(\langle \delta \tilde y \rangle^{-1}\phi \right)\right), \partial^{\alpha}(\langle \delta \tilde y \rangle^{-1}\phi)\right)_{L^2(\mathbb R^2)}\\
    %&=\sum_{0<|\alpha|\le4}\left(\partial^{\alpha}\left(\partial_y(1 + g)^2 \partial_y \left(\langle \delta y \rangle^{-1}\phi \right)\right), \partial^{\alpha}(\langle \delta y \rangle^{-1}\phi)\right)_{L^2(\mathbb R^2)}\\
    &=\sum_{0<|\alpha|\le4}\sum_{\beta\le\alpha}\frac{\alpha!}{\beta!(\alpha-\beta)!}\left(\partial^{\beta}\partial_y(1 + g)^2 \partial^{\alpha-\beta}\partial_y \left(\langle \delta \tilde y \rangle^{-1}\phi \right), \partial^{\alpha}(\langle \delta \tilde y \rangle^{-1}\phi)\right)_{L^2(\mathbb R^2)},
\end{align*}
due to the Sobolev embedding $H^{1+}(\mathbb R^2)\subset L^\infty(\mathbb R^2)$, we divide the scalar product in the above line into the following two cases
\begin{align*}
    &\left|\left(\partial^{\beta}\partial_y(1 + g)^2 \partial^{\alpha-\beta}\partial_y \left(\langle \delta \tilde y \rangle^{-1}\phi \right), \partial^{\alpha}(\langle \delta \tilde y \rangle^{-1}\phi)\right)_{L^2(\mathbb R^2)}\right|\\
    &\le\left\{
\begin{array}{ll}
   \left\|\partial^{\beta}\partial_y(1 + g)^2\right\|_{L^\infty(\mathbb R^2)} \left\|\partial^{\alpha-\beta}\partial_y \left(\langle \delta \tilde y \rangle^{-1}\phi \right)\right\|_{L^2(\mathbb R^2)}\left\| \partial^{\alpha}(\langle \delta \tilde y \rangle^{-1}\phi)\right\|_{L^2(\mathbb R^2)}, \ |\beta|\le1,\\
   \\
   \left\|\partial^{\beta}\partial_y(1 + g)^2\right\|_{L^2(\mathbb R^2)} \left\|\partial^{\alpha-\beta}\partial_y \left(\langle \delta \tilde y \rangle^{-1}\phi \right)\right\|_{L^\infty(\mathbb R^2)}\left\| \partial^{\alpha}(\langle \delta \tilde y \rangle^{-1}\phi)\right\|_{L^2(\mathbb R^2)}, \ 2
   \le|\beta|\le4,
\end{array}
\right.
\end{align*}
then follows from \eqref{ineq1} that 
\begin{align*}
    &\left|\sum_{0<|\alpha|\le4}\left(\partial^{\alpha}\left(\partial_y(1 + g)^2 \partial_y \left(\langle \delta \tilde y \rangle^{-1}\phi \right)\right), \partial^{\alpha}(\langle \delta \tilde y \rangle^{-1}\phi)\right)_{L^2(\mathbb R^2)}\right|\\
    &\le\frac{1}{32}\left\|\partial_y\left(\langle \delta \tilde y \rangle^{-1}\phi\right)\right\|^2_{H^4(\mathbb R^2)}+\tilde C_0\|g\|^4_{L^\infty([0, T]; H^4(\mathbb{R}^2))}\left\|\langle \delta \tilde y \rangle^{-1}\phi\right\|^2_{H^4}.
\end{align*}

For the term $I_{2, 1, 2}$, we first rewrite it as follows, 
\begin{align*}
    &\left(\partial^{\alpha}\left((1 + g)^2 \partial_y^2 \left(\langle \delta \tilde y \rangle^{-1}\phi \right)\right), \langle \delta \tilde y \rangle^{-1}\partial^\beta_y\langle \delta \tilde y \rangle\partial^{\alpha_1}_x\partial^{\alpha_2-\beta}_y(\langle \delta \tilde y \rangle^{-1}\phi)\right)_{L^2(\mathbb R^2)}\\
    &=-\left(\partial^{\alpha}\left((1 + g)^2 \partial_y \left(\langle \delta \tilde y \rangle^{-1}\phi \right)\right), \partial_y\left(\langle \delta \tilde y \rangle^{-1}\partial^\beta_y\langle \delta \tilde y \rangle\partial^{\alpha_1}_x\partial^{\alpha_2-\beta}_y(\langle \delta \tilde y \rangle^{-1}\phi)\right)\right)_{L^2(\mathbb R^2)}\\
    &\qquad -\left(\partial^{\alpha}\left(\partial_y(1 + g)^2 \partial_y \left(\langle \delta \tilde y \rangle^{-1}\phi \right)\right), \langle \delta \tilde y \rangle^{-1}\partial^\beta_y\langle \delta \tilde y \rangle\partial^{\alpha_1}_x\partial^{\alpha_2-\beta}_y(\langle \delta \tilde y \rangle^{-1}\phi)\right)_{L^2(\mathbb R^2)}.
\end{align*}
Since $\left\|g\right\|_{L^{\infty}([0, T]; H^{4}(\mathbb R^{2}))}\le\varepsilon_0$ and $\left|\partial_y\left(\langle \delta \tilde y \rangle^{-1}\partial^\beta_y\langle \delta \tilde y \rangle\right)\right|+\left|\langle \delta \tilde y \rangle^{-1}\partial^\beta_y\langle \delta \tilde y \rangle\right|\lesssim1$ for $\beta\ge1$, one has
\begin{align*}
    &\left|\left(\partial^{\alpha}\left((1 + g)^2 \partial_y \left(\langle \delta \tilde y \rangle^{-1}\phi \right)\right), \partial_y\left(\langle \delta \tilde y \rangle^{-1}\partial^\beta_y\langle \delta \tilde y \rangle\partial^{\alpha_1}_x\partial^{\alpha_2-\beta}_y(\langle \delta \tilde y \rangle^{-1}\phi)\right)\right)_{L^2(\mathbb R^2)}\right|\\
    &\le\sum_{\sigma\le\alpha}\frac{\alpha!}{\sigma!(\alpha-\sigma)!}\left|\left(\partial^{\sigma}(1 + g)^2 \partial_y \partial^{\alpha-\sigma}\left(\langle \delta \tilde y \rangle^{-1}\phi \right), \partial_y\left(\langle \delta \tilde y \rangle^{-1}\partial^\beta_y\langle \delta \tilde y \rangle\right)\partial^{\alpha_1}_x\partial^{\alpha_2-\beta}_y(\langle \delta \tilde y \rangle^{-1}\phi)\right)_{L^2(\mathbb R^2)}\right|\\
    &\qquad+\sum_{\sigma\le\alpha}\frac{\alpha!}{\sigma!(\alpha-\sigma)!}\left|\left(\partial^{\sigma}(1 + g)^2 \partial_y \partial^{\alpha-\sigma}\left(\langle \delta \tilde y \rangle^{-1}\phi \right), \langle \delta \tilde y \rangle^{-1}\partial^\beta_y\langle \delta \tilde y \rangle\partial^{\alpha_1}_x\partial^{\alpha_2-\beta+1}_y(\langle \delta \tilde y \rangle^{-1}\phi))\right)_{L^2(\mathbb R^2)}\right|\\
    &\le C_4\left\{
\begin{array}{ll}
	\sum_{\sigma\le\alpha}\left\|\partial^{\sigma}(1 + g)^2\right\|_{L^\infty(\mathbb R^2)}\|\partial_y \partial^{\alpha-\sigma}\left(\langle \delta \tilde y \rangle^{-1}\phi \right)\|_{L^2(\mathbb R^2)}\\
    \qquad\qquad\times\left(\|\partial^{\alpha_1}_x\partial^{\alpha_2-\beta}_y(\langle \delta \tilde y \rangle^{-1}\phi)\|_{L^2(\mathbb R^2)}+\|\partial^{\alpha_1}_x\partial^{\alpha_2-\beta+1}_y(\langle \delta \tilde y \rangle^{-1}\phi)\|_{L^2(\mathbb R^2)}\right),\quad |\sigma|\le2,\\
    \\
	\sum_{\sigma\le\alpha}\left\|\partial^{\sigma}(1 + g)^2\right\|_{L^2(\mathbb R^2)}\|\partial_y \partial^{\alpha-\sigma}\left(\langle \delta \tilde y \rangle^{-1}\phi \right)\|_{L^\infty(\mathbb R^2)}\\
    \qquad\qquad\times\left(\|\partial^{\alpha_1}_x\partial^{\alpha_2-\beta}_y(\langle \delta \tilde y \rangle^{-1}\phi)\|_{L^2(\mathbb R^2)}+\|\partial^{\alpha_1}_x\partial^{\alpha_2-\beta+1}_y(\langle \delta \tilde y \rangle^{-1}\phi)\|_{L^2(\mathbb R^2)}\right), \quad 3\le|\sigma|\le4.
\end{array}
\right.
\end{align*}
For $|\alpha|\ne0$, there exists there exist a $\gamma\in\mathbb N^2$ such that $|\gamma|=1$ and $\gamma\le\alpha$, so that we have
\begin{align*}
    &\left|\left(\partial^{\alpha}\left(\partial_y(1 + g)^2 \partial_y \left(\langle \delta \tilde y \rangle^{-1}\phi \right)\right), \langle \delta \tilde y \rangle^{-1}\partial^\beta_y\langle \delta \tilde y \rangle\partial^{\alpha_1}_x\partial^{\alpha_2-\beta}_y(\langle \delta \tilde y \rangle^{-1}\phi)\right)_{L^2(\mathbb R^2)}\right|\\
    &=\left|\left(\partial^{\alpha-\gamma}\left(\partial_y(1 + g)^2 \partial_y \left(\langle \delta \tilde y \rangle^{-1}\phi \right)\right), \partial^\gamma\left(\langle \delta \tilde y \rangle^{-1}\partial^\beta_y\langle \delta \tilde y \rangle\partial^{\alpha_1}_x\partial^{\alpha_2-\beta}_y(\langle \delta \tilde y \rangle^{-1}\phi)\right)\right)_{L^2(\mathbb R^2)}\right|\\
    &\le\sum_{\sigma\le\alpha-\gamma}\frac{(\alpha-\gamma)!}{\sigma!(\alpha-\gamma-\sigma)!}\left|\left(\partial^\sigma\partial_y(1 + g)^2\partial^{\alpha-\gamma-\sigma}\partial_y \left(\langle \delta \tilde y \rangle^{-1}\phi \right), \partial^\gamma\left(\langle \delta \tilde y \rangle^{-1}\partial^\beta_y\langle \delta \tilde y \rangle\partial^{\alpha_1}_x\partial^{\alpha_2-\beta}_y(\langle \delta \tilde y \rangle^{-1}\phi)\right)\right)_{L^2(\mathbb R^2)}\right|\\
    &\le C_4\left\{
\begin{array}{ll}
	\sum_{\sigma\le\alpha}\left\|\partial^{\sigma}(1 + g)^2\right\|_{L^\infty(\mathbb R^2)}\|\partial_y \partial^{\alpha-\sigma}\left(\langle \delta \tilde y \rangle^{-1}\phi \right)\|_{L^2(\mathbb R^2)}\\
    \qquad\qquad\times\left(\|\partial^{\alpha_1}_x\partial^{\alpha_2-\beta}_y(\langle \delta \tilde y \rangle^{-1}\phi)\|_{L^2(\mathbb R^2)}+\|\partial^{\alpha_1}_x\partial^{\alpha_2-\beta+1}_y(\langle \delta \tilde y \rangle^{-1}\phi)\|_{L^2(\mathbb R^2)}\right),\quad |\sigma|\le2,\\
    \\
	\sum_{\sigma\le\alpha}\left\|\partial^{\sigma}(1 + g)^2\right\|_{L^2(\mathbb R^2)}\|\partial_y \partial^{\alpha-\sigma}\left(\langle \delta \tilde y \rangle^{-1}\phi \right)\|_{L^\infty(\mathbb R^2)}\\
    \qquad\qquad\times\left(\|\partial^{\alpha_1}_x\partial^{\alpha_2-\beta}_y(\langle \delta \tilde y \rangle^{-1}\phi)\|_{L^2(\mathbb R^2)}+\|\partial^{\alpha_1}_x\partial^{\alpha_2-\beta+1}_y(\langle \delta \tilde y \rangle^{-1}\phi)\|_{L^2(\mathbb R^2)}\right), \quad 3\le|\sigma|\le4,
\end{array}
\right.
\end{align*}
here we also use $\left|\partial_y\left(\langle \delta \tilde y \rangle^{-1}\partial^\beta_y\langle \delta \tilde y \rangle\right)\right|+\left|\langle \delta \tilde y \rangle^{-1}\partial^\beta_y\langle \delta \tilde y \rangle\right|\lesssim1$ for all $|\beta|\ge1$. By Sobolev embedding theorem, we have
\begin{align*}
    |I_{2, 1, 2}|&\leq\left|\sum_{|\alpha|\le4}\sum_{0<\beta\le\alpha_2}\frac{\alpha_2!}{\beta!(\alpha_2-\beta)!}\left(\partial^{\alpha}\left((1 + g)^2 \partial_y^2 \left(\langle \delta \tilde y \rangle^{-1}\phi \right)\right), \langle \delta \tilde y \rangle^{-1}\partial^\beta_y\langle \delta \tilde y \rangle\partial^{\alpha_1}_x\partial^{\alpha_2-\beta}_y(\langle \delta \tilde y \rangle^{-1}\phi)\right)_{L^2(\mathbb R^2)}\right|\\
    &\le\tilde C_4\|g\|^2_{L^\infty([0, T]; H^4(\mathbb{R}^2))}\left\|\partial_y \left(\langle \delta \tilde y \rangle^{-1}\phi \right)\right\|_{H^4}\left\|\langle \delta \tilde y \rangle^{-1}\phi\right\|_{H^4}\\
    &\le\frac{1}{16}\left\|\partial_y\left(\langle \delta \tilde y \rangle^{-1}\phi\right)\right\|^2_{H^4(\mathbb R^2)}+4(\tilde C_4)^2\|g\|^4_{L^\infty([0, T]; H^4(\mathbb{R}^2))}\left\|\langle \delta \tilde y \rangle^{-1}\phi\right\|^2_{H^4}.
\end{align*}
Similarly, for the term $I_{2, 2}$, we can deduce that 
\begin{align*}
    \left|I_{2, 2}\right|\le \frac{1}{16}\left\|\partial_y\left(\langle \delta \tilde y \rangle^{-1}\phi\right)\right\|^2_{H^4(\mathbb R^2)}+4(C_5)^2\|g\|^4_{L^\infty([0, T]; H^4(\mathbb{R}^2))}\left\|\langle \delta \tilde y \rangle^{-1}\phi\right\|^2_{H^4}
\end{align*}
Combining all the estimates of $I_1$ and $I_2$, we have 
\begin{align*}
    &\left|\left(\phi, P\phi\right)_{H^4(\mathbb R^2)}\right|
    \ge-\frac12\frac{d}{dt}\left\|\phi(t)\right\|^2_{H^4(\mathbb R^2)}+\frac{1}{16}\left\|\partial_y\left(\langle \delta \tilde y \rangle^{-1}\phi\right)\right\|^2_{H^4(\mathbb R^2)}\\
    &\quad-\left\{4\left[\tilde C_0+(\tilde C_4)^2+(C_5)^2\right]\|g\|^4_{L^\infty([0, T]; H^4(\mathbb{R}^2))}+\tilde C_1+\tilde C_2+C_3\right\}\left\|\langle \delta \tilde y \rangle^{-1}\phi\right\|^2_{H^4},
\end{align*}
then
\begin{align*}
    &-\frac12\frac{d}{dt}\left\|\phi(t)\right\|^2_{H^4(\mathbb R^2)}+\frac{1}{16}\left\|\partial_y\left(\langle \delta \tilde y \rangle^{-1}\phi\right)\right\|^2_{H^4(\mathbb R^2)}\le\tilde C_5\left\|\phi\right\|^2_{H^4}+\left\|\phi\right\|_{H^4(\mathbb R^2)} \left\|P\phi\right\|_{H^4(\mathbb R^2)},
\end{align*}
where $\tilde C_5=4\left[\tilde C_0+(\tilde C_4)^2+(C_5)^2\right]+\tilde C_1+\tilde C_2+C_3$, this is equivalent to
\begin{align*}
    &-\frac{d}{dt}\left(e^{2\tilde C_5t}\left\|\phi(t)\right\|^2_{H^4(\mathbb R^2)}\right)+\frac18 e^{2\tilde C_5t}\left\|\partial_y\left(\langle \delta \tilde y \rangle^{-1}\phi\right)\right\|^2_{H^4(\mathbb R^2)}\le2e^{2\tilde C_5t}\left\|\phi\right\|_{H^4(\mathbb R^2)} \left\|P\phi\right\|_{H^4(\mathbb R^2)}.
\end{align*}
Integrating the above inequality from $t$ to $T$, since $\partial^\alpha\phi(T)=0, \forall \alpha\in\mathbb N^2$, we have
\begin{align*}
    &\left\|\phi(t)\right\|^2_{H^4(\mathbb R^2)}+\frac18\int_t^Te^{2\tilde C_5(s-t)}\left\|\partial_y\left(\langle \delta \tilde y \rangle^{-1}\phi\right)(s)\right\|^2_{H^4(\mathbb R^2)}ds\\
    &\le2e^{2\tilde C_5T}\int_t^T\left\|\phi(s)\right\|_{H^4(\mathbb R^2)} \left\|P\phi(s)\right\|_{H^4(\mathbb R^2)}ds,
\end{align*}
which implies 
\begin{align}\label{bound1}
    &\left\|\phi\right\|_{L^\infty([0, T]; H^4(\mathbb R^2))}\le2e^{2\tilde C_5T}\left\|P\phi\right\|_{L^1([0, T]; H^4(\mathbb R^2))}.
\end{align}
Let 
$$E=\left\{P\phi: \ \phi\in C^\infty([0, T]; C_c^\infty(\mathbb R^2)),\ \partial^\alpha\phi(T)=0, \forall \alpha\in\mathbb N^2 \right\}\subset L^1([0, T]; H^4(\mathbb R^2)).$$
Define a linear functioinal 
\begin{align*}
    &\mathcal T: E\mapsto \mathbb R\\
    &F=P\phi\to \left(v_{\delta}(0), \phi(0)\right)_{H^4(\mathbb R^2)}.
\end{align*}
If $P\phi_1=P\phi_2$ with $\phi_1, \phi_2\in C^\infty([0, T]; C_c^\infty(\mathbb R^2))$ and $\partial^\alpha\phi_1(T)=\partial^\alpha\phi_2(T)=0$, then from \eqref{bound1}
$$\left\|\phi_1-\phi_2\right\|_{L^\infty([0, T]; H^4(\mathbb R^2))}\le2e^{2\tilde C_5T}\left\|P(\phi_1-\phi_2)\right\|_{L^1([0, T]; H^4(\mathbb R^2))}=0.$$
Hence, the operator $P$ is injective. Therefore,
the linear functional $\mathcal T$ is well-defined,
\begin{align*}
   \left|\mathcal T(P\phi)\right|\le\|v_\delta(0)\|_{H^4}\left\|\phi\right\|_{L^\infty([0, T]; H^4(\mathbb R^2))}\le2e^{2\tilde C_5T}\|v_\delta(0)\|_{H^4}\left\|P\phi\right\|_{L^1([0, T]; H^4(\mathbb R^2))}.
\end{align*}
Thus, $\mathcal T$ is bounded on $\left(E, L^1([0, T]; H^4(\mathbb R^2))\right)$. From the Hahn-Banach extension Theorem, the operator $\mathcal T$ can be extended to a linear functional $\tilde{\mathcal T}$ on $L^1([0, T]; H^4(\mathbb R^2))$ such that
\begin{align*}
    \left|\tilde{\mathcal T}(F)\right|\le2e^{2\tilde C_5T}\|v_\delta(0)\|_{H^4}\left\|F\right\|_{L^1([0, T]; H^4(\mathbb R^2))}, \quad \forall F\in L^1([0, T]; H^4(\mathbb R^2)).
\end{align*}
So that in the end, it follows from the Riesz representation theorem that there exists a unique weak function $v_\delta\in L^\infty([0, T]; H^4(\mathbb R^2))$ of the Cauchy problem \eqref{equ-m}.

\bigskip

\textbf{Uniform bounded:} \ Taking scalar product in $H^4(\mathbb{R}^2)$ with $v_\delta$ on both side of \eqref{equ-m}, we have
\begin{align*}
    \frac{1}{2}\frac{d}{dt}\|v_\delta\|^2_{H^4(\mathbb{R}^2)}+\left( \langle \delta \tilde y \rangle^{-1}y \partial_x \left(\langle \delta \tilde y \rangle^{-1}v_\delta\right), v_\delta\right)_{H^4(\mathbb{R}^2)} - \left(\langle \delta \tilde y \rangle^{-1}(1 + g)^2 \partial_y^2 \left(\langle \delta \tilde y \rangle^{-1}v_\delta \right), v_\delta\right)_{H^4(\mathbb{R}^2)}= 0.
\end{align*}
We observe that $\langle\delta \tilde y \rangle^{-1}v_\delta\to0$ as $|x|, |y|\to\infty$, from this point, the proof is essentially the same as in the previous term $I_1$ and $I_2$, and we can deduce 
\begin{align*}
    \left( \langle \delta \tilde y \rangle^{-1}y \partial_x \left(\langle \delta \tilde y \rangle^{-1}v_\delta\right), v_\delta\right)_{H^4(\mathbb{R}^2)}\le C_6\left\|\langle \delta \tilde y \rangle^{-1}v_\delta\right\|^2_{H^4(\mathbb R^2)},
\end{align*}
and
\begin{align*}
    -\left(\langle \delta \tilde y \rangle^{-1}(1 + g)^2 \partial_y^2 \left(\langle \delta \tilde y \rangle^{-1}v_\delta \right), v_\delta\right)_{H^4(\mathbb{R}^2)}\ge\frac{1}{8}\left\|\partial_y\left(\langle \delta \tilde y \rangle^{-1}v_\delta\right)\right\|^2_{H^4(\mathbb R^2)}-C_7\left\|\langle \delta \tilde y \rangle^{-1}v_\delta\right\|^2_{H^4(\mathbb R^2)},
\end{align*}
thus
\begin{align*}
    \frac{d}{dt}\|v_\delta\|^2_{H^4(\mathbb{R}^2)} + \frac{1}{4}\left\|\partial_y\left(\langle \delta \tilde y \rangle^{-1}v_\delta\right)\right\|^2_{H^4(\mathbb R^2)} &\le 2(C_6+C_7)\left\|\langle \delta \tilde y \rangle^{-1}v_\delta\right\|^2_{H^4(\mathbb R^2)}\\
    &\le 2(C_6+C_7)\left\|v_\delta\right\|^2_{H^4(\mathbb R^2)}, 
\end{align*}
We now integrate the above inequality from 0 to $t$ to obtain: 
\begin{align*}
    \left\|v_\delta(t)\right\|^2_{H^4(\mathbb{R}^2)} + \frac{1}{4}\int_0^t\left\|\partial_y\left(\langle \delta \tilde y \rangle^{-1}v_\delta\right)(s)\right\|^2_{H^4(\mathbb R^2)}ds\le 2(C_6+C_7)\int_0^t\left\|v_\delta(s)\right\|^2_{H^4(\mathbb R^2)}ds+\left\|v_\delta(0)\right\|^2_{H^4(\mathbb{R}^2)}.
\end{align*}
By Gronwall's inequality, we have 
\begin{align*}
    \left\|v_\delta(t)\right\|^2_{H^4(\mathbb{R}^2)} \le 
    \left(1+2(C_6+C_7)Te^{2(C_6+C_7)T}\right)\left\|v_\delta(0)\right\|^2_{H^4(\mathbb{R}^2)}, \qquad \forall \ t\in[0, T],
\end{align*}
which leads to for all $0<t\le T$
\begin{align*}%\label{uniform bounded}
%\begin{split}
    \left\|v_\delta(t)\right\|^2_{H^4(\mathbb{R}^2)} + \frac{1}{4}\int_0^t\left\|\partial_y\left(\langle \delta \tilde y \rangle^{-1}v_\delta\right)(s)\right\|^2_{H^4(\mathbb R^2)}ds&\le \left(1+2(C_6+C_7)Te^{2(C_6+C_7)T}\right)^2\left\|v_\delta(0)\right\|^2_{H^4(\mathbb{R}^2)}\\
    &\le B_0\left\|v_0\right\|^2_{H^4(\mathbb{R}^2)},
%\end{split}    
\end{align*}
where the constant $B_0=\left(1+2(C_6+C_7)Te^{2(C_6+C_7)T}\right)^2$ is independent of $0<\delta\ll1$.
\end{proof}

\begin{remark}\label{C-infty}
If there exists a sufficiently small constant $\varepsilon_0>0$ such that 
\begin{align*}%\label{g small} 
     \left\|g\right\|_{C^{\infty}([0, T]; H^{4}(\mathbb R^{2}))}\le\varepsilon_0,\ \   \ \ \ \ 
     \|v_0\|_{H^4(\mathbb R^{2})}\le \varepsilon_0.
\end{align*}
We can show the $C^k$ in time by using induction on the index $k \in \mathbb{N}$. %This implies that the Cauchy problem \eqref{equ-m} admits a unique weak solution in $C^\infty([0, T]; H^4(\mathbb R^2))$. For the case of $k=0$, we have already proved. Assume that it is true up to the case of $k-1$. Then for the case of $k$, we have 
%\begin{align}
    %&\frac{1}{2}\frac{d}{dt}\|v_\delta^k\|^2_{H^4(\mathbb{R}^2)}+\left( \langle \delta \tilde y \rangle^{-1}y \partial_x \left(\langle \delta \tilde y \rangle^{-1}v_\delta^k\right), v_\delta^k\right)_{H^4(\mathbb{R}^2)} - \left(\langle \delta \tilde y \rangle^{-1}(1 + g)^2 \partial_y^2 \left(\langle \delta \tilde y \rangle^{-1}v_\delta^k \right), v_\delta^k\right)_{H^4(\mathbb{R}^2)}\\
    %&\qquad - \sum_{j=1}^k C_k^j\left(\langle \delta \tilde y \rangle^{-1}\partial_t^j(1 + g)^2 \partial_y^2 \left(\langle \delta \tilde y \rangle^{-1}v_\delta^{k-j} \right), v_\delta^k\right)_{H^4(\mathbb{R}^2)} = 0,
%\end{align}
%where $v_\delta^k=\partial^k_t v_\delta$. 
\end{remark}

\section{Local existence and energy estimate}\label{existence}

To obtain the existence of the original system \eqref{2-1}, we first consider the following linearized equation
\begin{equation}\label{2-2}
\begin{cases}
\partial_t v + y\partial_x v - (1+g)^2 \partial_{y}^2 v = 0, \ \  (x, y) \in \mathbb{R}^2,\ \  t > 0, \\
v \big|_{t = 0} = v_0(x, y). 
\end{cases}
\end{equation}
Due to the uniform bounded of modified equation, we may assume the initial datum $v_0$ and $g$ satisfy \eqref{g small}. We will establish the existence of \eqref{2-2}, by taking the limit as $\delta\to0+$, utilizing the uniform estimates established in Proposition \ref{prop2.1-0}. 
We first give a basic property of reflexive spaces.
\begin{lemma}(\cite{Brezis})(Kakutani Theorem)\label{Kakutani}
    Let $X$ be a Banach space. Then $X$ is reflexive if and only if 
    $$B_E=\{x\in X: \|x\|_X\le1\},$$
    is compact in the weak topology $(X, X^*)$.
\end{lemma}
We now state the following convergence result.
\begin{proposition}\label{prop2.1}
For any $T > 0$ and $v_0, g \in H^4(\mathbb{R}^2)$ satisfies \eqref{g small}, then the Cauchy problem \eqref{2-2} admits a unique weak solution $v\in L^\infty(]0, T]; H^4(\mathbb R^2))$. Moreover, we have
\begin{align}\label{2-0} 
     &\left\|v(t)\right\|^{2}_{H^{4}(\mathbb R^{2})}+\frac14\int_{0}^{t}\left\| \partial_y v(s)\right\|_{H^4(\mathbb R^{2})}^2ds\le {B_0} \left\|v_{0}\right\|^{2}_{H^{4}(\mathbb R^{2})},\ \ \ 0<t\le T.
\end{align}
\end{proposition}

\begin{proof}
We use the uniform bounded estimates \eqref{uniform bounded} to obtain local strong (smooth) solutions to the original system by taking the limit as $\delta\to0+$. 

Since the constant 
$B_0$ is independent of $0<\delta\ll1$, from \eqref{uniform bounded}, since $H^4(\mathbb R^2)$ is reflexive, together with Lemma \ref{Kakutani} (Kakutani theorem), then $v_\delta$ is compact in the weak topology on $H^4(\mathbb R^2)$, thus for any $0< t \le T$, there exists a $v(t)\in H^4(\mathbb R^2)$ such that 
$$v_{\delta}(t)\to v(t) \quad \text{weakly in} \quad  H^4(\mathbb R^2).$$
%\ \left(v_{\delta}\to v \quad \text{weakly-star in} \quad  L^\infty(]0, T]; H^4(\mathbb R^2)\right)$$
Similarly, we have
$$\partial_y\left(\langle \delta \tilde y \rangle^{-1}v_\delta\right) \to \partial_y v \quad \text{weakly in} \quad L^2(]0, t]; H^4(\mathbb R^2)).$$
%that is for any $\Gamma_1\in H^4(\mathbb R^2)$ and %$\Gamma_2\in L^2([0, t]; H^4(\mathbb R^2))$ such that 
    %$$\lim_{\delta\to0+}\left\langle \Gamma_1, v_\delta(t) \right\rangle=\langle \Gamma_1, v(t) \rangle \quad \lim_{\delta\to0+}\left\langle \Gamma_2, \partial_y\left(\langle \delta y \rangle^{-1}v_\delta\right) \right\rangle=\langle \Gamma_2, \partial_y v \rangle.$$
Moreover, it follows from \eqref{uniform bounded}  that
\begin{align*}%\label{uniform bounded}
%\begin{split}
    &\left\|v(t)\right\|^{2}_{H^{4}(\mathbb R^{2})}+\frac14\int_{0}^{t}\left\| \partial_y v(s)\right\|_{H^4(\mathbb R^{2})}^2ds\\
    &\le\liminf_{\delta\to0+}\left(\left\|v_\delta(t)\right\|^2_{H^4(\mathbb{R}^2)} + \frac{1}{4}\int_0^t\left\|\partial_y\left(\langle \delta \tilde y \rangle^{-1}v_\delta\right)(s)\right\|^2_{H^4(\mathbb R^2)}ds\right)\\
    &\le B_0\left\|v_0\right\|^2_{H^4(\mathbb{R}^2)}, \quad 0<t\le T.
%\end{split}    
\end{align*}

We now show that the limit $v$ satisfies the equation \eqref{2-2} in the sense of distributions. For any $\varphi\in C^\infty_c([0, T]\times\mathbb R^2)$, we estimate
\begin{align*}
    0&=\int_{\mathbb R^2}v(t)\varphi(t)dxdy-\int_{\mathbb R^2}v_0\varphi(0)dxdy+\int_0^t\int_{\mathbb R^2}[P(v)](s, x, y)\varphi(s, x, y)dxdyds\\
    %&\quad-\int_0^t\int_{\mathbb R^2}(1+g)^2\partial^2_yv(s, x, y)\varphi(s, x, y)dsdxdy\\
    &=\lim_{\delta\to0+}\left\{\int_{\mathbb R^2}v_\delta(t)\varphi(t)dxdy-\int_{\mathbb R^2}v_\delta(0)\varphi(0)dxdy+\int_0^t\int_{\mathbb R^2}[P^\delta(v_\delta)](s, x, y)\varphi(s, x, y)dxdyds\right\},
\end{align*}
where 
$$P(v)=y\partial_xv-(1+g)^2\partial^2_yv,$$ 
and 
$$P^\delta(v_\delta)=\langle \delta \tilde y \rangle^{-1}y \partial_x \left(\langle \delta \tilde y \rangle^{-1}v_\delta\right) - \langle \delta \tilde y \rangle^{-1}(1 + g)^2 \partial_y^2 \left(\langle \delta \tilde y \rangle^{-1}v_\delta \right).$$

By using the fact that $v_{\delta}(t)\to v(t)$ weakly in $ H^4(\mathbb R^2)$, it is obvious that
\begin{align*}
    (v, \varphi)_{L^2}=\lim_{\delta\to0+}(v_\delta, \varphi)_{L^2}, \quad (v_0, \varphi)_{L^2}=\lim_{\delta\to0+}(v_\delta(0), \varphi)_{L^2}.
\end{align*}
From integration by parts, one has
\begin{align*}
    &\int_0^t\int_{\mathbb R^2}\langle \delta \tilde y \rangle^{-1}y \partial_x \left(\langle \delta \tilde y \rangle^{-1}v_\delta\right)\varphi dxdyds-\int_0^t\int_{\mathbb R^2}y \partial_x v \varphi dxdyds\\
    &=\int_0^t\int_{\mathbb R^2}y \partial_x v_\delta \langle \delta \tilde y \rangle^{-2}\varphi dxdyds-\int_0^t\int_{\mathbb R^2}y \partial_x v \varphi dxdyds\\
    &=\int_0^t\int_{\mathbb R^2}y \partial_x v_\delta\left(\langle \delta \tilde y \rangle^{-2}\varphi - \varphi\right)dxdyds + \int_0^t\int_{\mathbb R^2}y \partial_x\left(v_\delta - v \right) \varphi dxdyds\\
    &=-\int_0^t\int_{\mathbb R^2} v_\delta y \partial_x\left(\langle \delta \tilde y \rangle^{-2}\varphi - \varphi\right)dxdyds - \int_0^t\int_{\mathbb R^2}\left(v_\delta - v \right) y \partial_x\varphi dxdyds.
\end{align*}
By using the uniform estimate \eqref{uniform bounded}, we have
\begin{align*}
    \left|\int_0^t\int_{\mathbb R^2} v_\delta y \partial_x\left(\langle \delta \tilde y \rangle^{-2}\varphi - \varphi\right)dxdyds\right|\le \|v_\delta\|_{L^\infty([0, T]; L^2)}\int_0^t\|y \partial_x\left(\langle \delta \tilde y \rangle^{-2}\varphi - \varphi\right)\|_{L^2}ds\to0, \quad \delta\to0+,
\end{align*}
and since $v_{\delta}(s)\to v(s)$ weakly in $ H^4(\mathbb R^2)$ for $0\le s\le t$ and $y \partial_x\varphi\in H^4(\mathbb R^2)$, we have
\begin{align*}
    \left|\int_0^t\int_{\mathbb R^2}\left(v_\delta - v \right) y \partial_x\varphi dxdyds\right|\to0, \quad \delta\to0+.
\end{align*}

It remains to show that 
\begin{align*}
    \lim_{\delta\to0+}\int_0^t\int_{\mathbb R^2}\langle \delta \tilde y \rangle^{-1}(1 + g)^2 \partial_y^2 \left(\langle \delta \tilde y \rangle^{-1}v_\delta \right)\varphi dxdyds=\int_0^t\int_{\mathbb R^2}(1 + g)^2 \partial_y^2 v \varphi dxdyds.
\end{align*}
We will consider it in the following two terms
\begin{align*}
    &\int_0^t\int_{\mathbb R^2}\langle \delta \tilde y \rangle^{-1}(1 + g)^2 \partial_y^2 \left(\langle \delta \tilde y \rangle^{-1}v_\delta \right)\varphi dxdyds-\int_0^t\int_{\mathbb R^2}(1 + g)^2 \partial_y^2 v \varphi dxdyds\\
    &=\int_0^t\int_{\mathbb R^2} (1 + g)^2 \partial_y^2 \left(\langle \delta \tilde y \rangle^{-1}v_\delta \right)\langle \delta \tilde y \rangle^{-1}\varphi dxdyds-\int_0^t\int_{\mathbb R^2}(1 + g)^2 \partial_y^2 v \varphi dxdyds\\
    &=\int_0^t\int_{\mathbb R^2} \partial_y^2 \left(\langle \delta \tilde y \rangle^{-1}v_\delta \right)(1 + g)^2\left(\langle \delta \tilde y \rangle^{-1}\varphi - \varphi\right)dxdyds + \int_0^t\int_{\mathbb R^2} \partial_y^2\left(\langle \delta \tilde y \rangle^{-1}v_\delta - v \right) (1+g)^2 \varphi dxdyds.
\end{align*}
From \eqref{uniform bounded} and \eqref{g small}, by taking the limit, we have
\begin{align*}
    &\left|\int_0^t\int_{\mathbb R^2} \partial_y^2 \left(\langle \delta \tilde y \rangle^{-1}v_\delta \right)(1 + g)^2\left(\langle \delta \tilde y \rangle^{-1}\varphi - \varphi\right)dxdyds\right|\\
    &\le\frac14\int_0^t\|\partial^2_y\left(\langle \delta \tilde y \rangle^{-1}v_\delta \right)\|^2_{L^2(\mathbb R^2)}ds+\int_0^t\|(1 + g)^2\left(\langle \delta \tilde y \rangle^{-1}\varphi - \varphi\right)\|^2_{L^2(\mathbb R^2)}\\
    &\le\frac14\int_0^t\|\partial_y\left(\langle \delta \tilde y \rangle^{-1}v_\delta \right)\|^2_{H^4(\mathbb R^2)}ds+\int_0^t\|(1 + g)^2\left(\langle \delta \tilde y \rangle^{-1}\varphi - \varphi\right)\|^2_{L^2(\mathbb R^2)}\to0, \quad \delta\to0+.
\end{align*}
Since $\varphi\in C^\infty_c([0,T]\times\mathbb R^2)$, one has $\partial_y\left[(1+g)^2 \varphi\right]\in L^2([0, T]; H^4(\mathbb R^2))$, then by applying the fact that $\partial_y\left(\langle \delta \tilde y \rangle^{-1}v_\delta\right) \to \partial_y v$ weakly in $L^2(]0, t]; H^4(\mathbb R^2))$, we have 
\begin{align*}
    &\int_0^t\int_{\mathbb R^2} \partial_y^2\left(\langle \delta \tilde y \rangle^{-1}v_\delta - v \right) (1+g)^2 \varphi dxdyds\\
    &=-\int_0^t\int_{\mathbb R^2} \partial_y\left(\langle \delta \tilde y \rangle^{-1}v_\delta - v \right) \partial_y\left[(1+g)^2 \varphi\right] dxdyds\to0, \quad \delta\to0+.
\end{align*}

We finally combine the convergence we have acquired so far, it follows that the limit $v$ satisfies the equation \eqref{2-2} in the sense of distributions.
\end{proof}

Based on the construction in Proposition \ref{prop2.1}, we know the existence of the Cauchy problem \eqref{2-1} when $v_0 \in H^4(\mathbb R^{2})$. Specifically, we have the following result.

\begin{proposition}\label{existence111}
 Let $v_0 \in H^4(\mathbb R^{2})$, and there exists a sufficiently small constant $\varepsilon > 0$ such that
\begin{align*}%\label{initial small} 
     \left\|v_{0}\right\|_{H^{4}(\mathbb R^{2})}\le\varepsilon,
\end{align*}   
then, there exists $T>0$, such that the Cauchy problem \eqref{2-1} admits a weak solution ${v\in L^{\infty}(]0, T]; H^{4}(\mathbb R^{2}))}$. Moreover, it satisfies 
\begin{align}\label{2-0-1} 
     &\left\|v(t)\right\|^{2}_{H^{4}(\mathbb R^{2})}+\frac14\int_{0}^{t}\left\| \partial_y v(s)\right\|_{H^4(\mathbb R^{2})}^2ds\le \tilde B_0 \varepsilon^{2},\ \ \ \ \ 0<t\le T.
\end{align}
\end{proposition}

\begin{remark}
    From the Remark \ref{C-infty}, we can obtain that ${v\in C^{\infty}(]0, T]; H^{4}(\mathbb R^{2}))}$.
\end{remark}

\begin{proof}
Let $v^{0}=v_{0}\in H^4(\mathbb R^{2})$, consider the following approximate equations  
$$
\partial_t v^{n} + y\partial_x v^{n} - (1 + v^{n-1})^2 \partial_{y}^2 v^{n} = 0,  \quad v^{n} \big|_{t = 0} = v_0, \quad n\ge1.
$$     
From iterating and Proposition \ref{prop2.1}, one can obtain that for all $n\ge1$, the approximate equations admit a solution $\{ v^{n}\}$ satisfying for all $0<t\le T$
\begin{align}\label{n small}
     &\left\|v^{n}(t)\right\|^{2}_{H^{4}(\mathbb R^{2})}+\frac14\int_{0}^{t}\left\| \partial_y v^{n}(s)\right\|_{H^4(\mathbb R^{2})}^2ds\le \tilde B_{0}\left\|v_{0}\right\|^{2}_{H^{4}(\mathbb R^{2})}\le \tilde B_{0} \varepsilon^{2}.
\end{align}

We need to prove the convergence of sequence $\{ v^{n}\}$ in $L^\infty (]0, T]; H^4(\mathbb{R}^2))$ and 
$\{\partial_y v^{n}\}$ in  $L^2([0, T]; H^4(\mathbb{R}^2))$.  
%From \eqref{n small}, for all $m, n\in\mathbb N$, one has
%\begin{align}\label{mn}
    %\left\|v^{n}(t)-v^m(t)\right\|_{H^{4}(\mathbb R^{2})}\le\left\|v^{n}(t)\right\|_{H^{4}(\mathbb R^{2})}+\left\|v^{m}(t)\right\|_{H^{4}(\mathbb R^{2})}\le\sqrt{2\tilde B_{0}}\varepsilon.
%\end{align}
Setting 
$${
\zeta^{n}(t)=v^{n+1}(t)-v^{n}(t)\in H^4(\mathbb R^2), \quad n\ge1, \quad t \in ]0, T], }
$$ 
then $\zeta$ satisfies the following equation
\begin{equation*}%\label{1-1}
\left\{
\begin{aligned}
 &\partial_t \zeta^{n} + y\partial_x \zeta^{n} - (1 + v^{n})^2 \partial_{y}^2 \zeta^{n} -  \left((1+v^{n-1})^2-(1+v^n)^2\right) \partial_{y}^2 v^{n}  = 0, \\
 &\zeta^{n} \big|_{t = 0} = 0, \  n\ge1.
\end{aligned}
\right.
\end{equation*}
Multiplying both sides of the above equation by $\langle \delta \tilde y\rangle^{-1}$ with $0<\delta\ll1$, then
\begin{align*}
    &\partial_t \left(\langle \delta \tilde y\rangle^{-1}\zeta^{n}\right) + y\partial_x \left(\langle \delta \tilde y\rangle^{-1}\zeta^{n}\right) - (1 + v^{n})^2 \partial_{y}^2 \left(\langle \delta \tilde y\rangle^{-1}\zeta^{n}\right) \\
    &\quad - \langle \delta \tilde y\rangle^{-1}\left((1+v^{n-1})^2-(1+v^n)^2\right) \partial_{y}^2 v^{n}  \\
    &= -(1 + v^{n})^2 \left[\partial_{y}^2, \langle \delta \tilde y\rangle^{-1}\right] \zeta^{n},
\end{align*}
taking inner product with $\langle \delta \tilde y\rangle^{-1}\zeta^{n}=\vartheta^n$ in $H^4(\mathbb R^2)$, then
\begin{align*} 
     &\frac12\frac{d}{dt}\left\|\vartheta^n(t)\right\|^{2}_{H^{4}(\mathbb R^{2})}+\left(y\partial_x \vartheta^n, \vartheta^n\right)_{H^{4}(\mathbb R^{2})}-\left((1+v^n)^2\partial_y^{2}\vartheta^n, \vartheta^n\right)_{H^{4}(\mathbb R^{2})}\\
     &\le\left|\left(\vartheta^{n-1}(2+v^{n-1}+v^n)\partial_y^{2}v^{n}, \vartheta^n\right)_{H^{4}(\mathbb R^{2})}\right|\\
     %&\quad+\left|\left(\zeta^{n-1}(2+v^{n-1}+v^n)\left[\partial_{y}^2, \langle \delta y\rangle^{-1}\right]v^{n}, \vartheta\right)_{H^4(\mathbb R^2)}\right|\\
     &\quad+\left|\left((1 + v^{n})^2 \left[\partial_{y}^2, \langle \delta \tilde y\rangle^{-1}\right] \zeta^{n}, \vartheta^n\right)_{H^4(\mathbb R^2)}\right|=|\Theta_1|+|\Theta_2|.
\end{align*}

We first consider $\left(y\partial_x \vartheta^n, \vartheta^n\right)_{H^{4}(\mathbb R^{2})}$. Using Leibniz formula, one has
\begin{align*}
    \left(y\partial_x \vartheta^n, \vartheta^n\right)_{H^{4}(\mathbb R^{2})}&=\sum_{|\alpha|\le4}\left(\partial^{\alpha_2}_y\left(y\partial^{\alpha_1+1}_x \vartheta^n\right), \partial^\alpha\vartheta^n\right)_{L^{2}(\mathbb R^{2})}\\
    &=\sum_{|\alpha|\le4}\left(y\partial_x\partial^{\alpha} \vartheta^n, \partial^\alpha\vartheta^n\right)_{L^{2}(\mathbb R^{2})}+\sum_{|\alpha|\le4, \alpha_2\ne0}\left(\partial^{\alpha_1+1}_x\partial^{\alpha_2-1}_y \vartheta^n, \partial^\alpha\vartheta^n\right)_{L^{2}(\mathbb R^{2})},
\end{align*}
since $\vartheta^n\to0$ as $|x|, |y|\to\infty$, together with integration by parts we can gain 
$$\sum_{|\alpha|\le4}\left(y\partial_x\partial^{\alpha} \vartheta^n, \partial^\alpha\vartheta^n\right)_{L^{2}(\mathbb R^{2})}=0.$$
Thus, we directly calculated that
$$\left(y\partial_x \vartheta^n, \vartheta^n\right)_{H^{4}(\mathbb R^{2})}\lesssim\|\vartheta^n(t)\|^2_{H^4(\mathbb R^2)}.$$

Moving the negative integral term on the left hand side
of in \eqref{n small}, we observe that 
\begin{align}\label{v-n small}
    \|v^n\|_{L^\infty(]0, T]; H^4(\mathbb R^2))}\le\sqrt{\tilde B_0}\varepsilon, \quad n\ge1,
\end{align}
as discussed in $I_{2}$, we can get 
\begin{align*}
    -\left((1+v^n)^2\partial_y^{2}\vartheta^n, \vartheta^n\right)_{H^{4}(\mathbb R^{2})}\ge\frac{1}{8}\left\|\partial_y\vartheta^n(t)\right\|^2_{H^4(\mathbb R^2)}-C_8\left\|\vartheta^n(t)\right\|^2_{H^4(\mathbb R^2)}.
\end{align*}
For the term $\Theta_1$, direct calculation yields
\begin{align*}
    \Theta_1&=-\left(\vartheta^{n-1}(2+v^{n-1}+v^n)\partial_yv^{n}, \partial_y\vartheta^n\right)_{H^{4}(\mathbb R^{2})}-\left(\partial_y\vartheta^{n-1}(2+v^{n-1}+v^n)\partial_yv^{n}, \vartheta^n\right)_{H^{4}(\mathbb R^{2})}\\
    &\quad-\left(\vartheta^{n-1}\partial_y(v^{n-1}+v^n)\partial_yv^{n}, \vartheta^n\right)_{H^{4}(\mathbb R^{2})}.
\end{align*}
Using the Cauchy-Schwarz inequality and \eqref{v-n small}, together with
\begin{align}\label{com}
    \|f g\|_{H^4(\mathbb R^2)}\le\|f\|_{H^4(\mathbb R^2)}\|g\|_{H^4(\mathbb R^2)}, \quad f, \ g\in H^4(\mathbb R^2),
\end{align}
we can get that 
\begin{align*}
    \left|\Theta_1\right|&\lesssim\left\|\vartheta^{n-1}\right\|_{H^4(\mathbb R^2)}\left\|\partial_yv^{n}\right\|_{H^4(\mathbb R^2)} \left\|\partial_y\vartheta^n\right\|_{H^{4}(\mathbb R^{2})}+\left\|\partial_y\vartheta^{n-1}\right\|_{H^4(\mathbb R^2)}\left\|\partial_yv^{n}\right\|_{H^4(\mathbb R^2)} \left\|\vartheta^n\right\|_{H^{4}(\mathbb R^{2})}\\
    &\quad+\left\|\vartheta^{n-1}\right\|_{H^4(\mathbb R^2)}\left\|\vartheta^n\right\|_{H^{4}(\mathbb R^{2})}\left(\left\|\partial_yv^{n-1}\right\|_{H^4(\mathbb R^2)}\left\|\partial_yv^{n}\right\|_{H^4(\mathbb R^2)} +\left\|\partial_yv^{n}\right\|^2_{H^4(\mathbb R^2)}\right).
\end{align*}
We now turn to consider the term $\Theta_2$. Since 
\begin{align*}
    \left[\partial_{y}^2, \langle \delta \tilde y\rangle^{-1}\right] \zeta^{n}=\partial_{y}^2\langle \delta \tilde y\rangle^{-1} \zeta^{n}+2\partial_{y}\langle \delta \tilde y\rangle^{-1} \partial_{y}\zeta^{n}=\langle \delta \tilde y\rangle\partial_{y}^2\langle \delta \tilde y\rangle^{-1} \vartheta^{n}+2\partial_{y}\langle \delta \tilde y\rangle^{-1} \partial_{y}\left(\langle \delta \tilde y\rangle\vartheta^{n}\right),
\end{align*}
then follows from \eqref{v-n small}, \eqref{com} and the Cauchy-Schwarz inequality that
\begin{align*}
    \left|\Theta_2\right|&\lesssim\left\| \left[\partial_{y}^2, \langle \delta \tilde y\rangle^{-1}\right] \zeta^{n}\right\|_{H^4(\mathbb R^2)}\left\| \vartheta^n\right\|_{H^4(\mathbb R^2)}\\
    &\lesssim\left\|\langle \delta \tilde y\rangle\partial_{y}^2\langle \delta \tilde y\rangle^{-1} \vartheta^{n}\right\|_{H^4(\mathbb R^2)}+\left\|\partial_{y}\langle \delta \tilde y\rangle^{-1} \partial_{y}\langle \delta \tilde y\rangle\vartheta^{n}\right\|_{H^4(\mathbb R^2)}+\left\|\langle \delta \tilde y\rangle\partial_{y}\langle \delta \tilde y\rangle^{-1} \partial_{y}\vartheta^{n}\right\|_{H^4(\mathbb R^2)}\\
    &\lesssim \delta\left(\left\| \vartheta^{n}\right\|_{H^4(\mathbb R^2)}+\left\| \partial_{y}\vartheta^{n}\right\|_{H^4(\mathbb R^2)}\right),
\end{align*}
the last line above we use the fact that $\left|\partial^j_y\langle \delta \tilde y\rangle^{-1}\right|\lesssim \delta^j\langle \delta \tilde y\rangle^{-1-j}$ for $j=1, 2$. We now combine all estimate into one that 
\begin{align*}
    &\frac12\frac{d}{dt}\left\|\vartheta^n(t)\right\|^{2}_{H^{4}(\mathbb R^{2})}+\frac{1}{8}\left\|\partial_y\vartheta^n(t)\right\|^2_{H^4(\mathbb R^2)}\le\tilde C_8\left\|\vartheta^n(t)\right\|^2_{H^4(\mathbb R^2)}+\tilde C_8\delta\left\| \partial_{y}\vartheta^{n}(t)\right\|_{H^4(\mathbb R^2)}\\
     &\quad+\tilde C_8\left\|\vartheta^{n-1}\right\|_{H^4(\mathbb R^2)}\left\|\partial_yv^{n}\right\|_{H^4(\mathbb R^2)} \left\|\partial_y\vartheta^n\right\|_{H^{4}(\mathbb R^{2})}+\tilde C_8\left\|\partial_y\vartheta^{n-1}\right\|_{H^4(\mathbb R^2)}\left\|\partial_yv^{n}\right\|_{H^4(\mathbb R^2)} \left\|\vartheta^n\right\|_{H^{4}(\mathbb R^{2})}\\
    &\quad+\tilde C_8\left\|\vartheta^{n-1}\right\|_{H^4(\mathbb R^2)}\left\|\vartheta^n\right\|_{H^{4}(\mathbb R^{2})}\left(\left\|\partial_yv^{n-1}\right\|_{H^4(\mathbb R^2)}\left\|\partial_yv^{n}\right\|_{H^4(\mathbb R^2)} +\left\|\partial_yv^{n}\right\|^2_{H^4(\mathbb R^2)}\right).
\end{align*}
Integrating from $0$ to $t$, it follows from H$\rm\ddot o$lder's inequality and \eqref{n small} that 
\begin{align*}
    &\frac12\left\|\vartheta^n(t)\right\|^{2}_{H^{4}(\mathbb R^{2})}+\frac{1}{8}\int_0^t\left\|\partial_y\vartheta^n(s)\right\|^2_{H^4(\mathbb R^2)}ds\\
    &\le\frac12\left\|\vartheta^n(0)\right\|^{2}_{H^{4}(\mathbb R^{2})}+\tilde C_8\varepsilon\left\|\vartheta^{n-1}\right\|_{L^\infty(]0, t]; H^4(\mathbb R^2))} \left\|\partial_y\vartheta^n\right\|_{L^2(]0, t]; H^{4}(\mathbb R^{2}))}\\
     &\quad+\tilde C_8\int_0^t\left\|\vartheta^n(s)\right\|^2_{H^4(\mathbb R^2)}ds+\tilde C_8\varepsilon\left\|\vartheta^{n}\right\|_{L^\infty(]0, t]; H^4(\mathbb R^2))}\left\|\partial_y\vartheta^{n-1}(s)\right\|_{L^2(]0, t]; H^4(\mathbb R^2))}\\
    &\quad+\tilde C_8\delta\int_0^t\left\| \partial_{y}\vartheta^{n}(s)\right\|_{H^4(\mathbb R^2)}ds+\tilde C_8\varepsilon^2\left\|\vartheta^{n-1}\right\|_{L^\infty(]0, t]; H^4(\mathbb R^2))}\left\|\vartheta^{n}\right\|_{L^\infty(]0, t]; H^4(\mathbb R^2))}\\
    &\le\frac12\left\|\vartheta^n(0)\right\|^{2}_{H^{4}(\mathbb R^{2})}+\tilde C_8\int_0^t\left\|\vartheta^n(s)\right\|^2_{H^4(\mathbb R^2)}ds+\left(\tilde C_8\delta+\frac{1}{32}\right)\int_0^t\left\| \partial_{y}\vartheta^{n}(s)\right\|_{H^4(\mathbb R^2)}ds\\
     &\quad+\left(4\left(\tilde C_8\varepsilon\right)^2+2\left(\tilde C_8\varepsilon^2\right)^2\right)\left\|\vartheta^{n-1}\right\|^2_{L^\infty(]0, t]; H^4(\mathbb R^2))}+\frac14\left\|\vartheta^{n}\right\|^2_{L^\infty(]0, t]; H^4(\mathbb R^2))}\\
     &\quad+2\left(\tilde C_8\varepsilon\right)^2\int_0^t\left\|\partial_y\vartheta^{n-1}(s)\right\|^2_{H^4(\mathbb R^2)}ds,
\end{align*}
taking $\tilde C_8\delta\le\frac{1}{32}$, then we have that for all $0<t\le T$
\begin{align*}
    &\left\|\vartheta^n(t)\right\|^{2}_{H^{4}(\mathbb R^{2})}+\frac{1}{8}\int_0^t\left\|\partial_y\vartheta^n(s)\right\|^2_{H^4(\mathbb R^2)}ds\\
    &\le\left\|\vartheta^n(0)\right\|^{2}_{H^{4}(\mathbb R^{2})}+2\tilde C_8\int_0^t\left\|\vartheta^n(s)\right\|^2_{H^4(\mathbb R^2)}ds+\frac12\left\|\vartheta^{n}\right\|^2_{L^\infty(]0, t]; H^4(\mathbb R^2))}\\
     &\quad+16\left(\tilde C_8\varepsilon\right)^2\left(\left\|\vartheta^{n-1}\right\|^2_{L^\infty(]0, t]; H^4(\mathbb R^2))}+\frac18\int_0^t\left\|\partial_y\vartheta^{n-1}(s)\right\|^2_{H^4(\mathbb R^2)}ds\right)\\
     &\le\left\|\vartheta^n(0)\right\|^{2}_{H^{4}(\mathbb R^{2})}+2\tilde C_8\int_0^T\left\|\vartheta^n(s)\right\|^2_{H^4(\mathbb R^2)}ds+\frac12\left\|\vartheta^{n}\right\|^2_{L^\infty(]0, T]; H^4(\mathbb R^2))}\\
     &\quad+16\left(\tilde C_8\varepsilon\right)^2\left(\left\|\vartheta^{n-1}\right\|^2_{L^\infty(]0, T]; H^4(\mathbb R^2))}+\frac18\int_0^T\left\|\partial_y\vartheta^{n-1}(s)\right\|^2_{H^4(\mathbb R^2)}ds\right)
\end{align*}
moving the negative integral term on the left hand side of above inequality, one can conclude 
\begin{align*}
    &\frac12\left\|\vartheta^{n}\right\|^2_{L^\infty(]0, T]; H^4(\mathbb R^2))}
    \le\left\|\vartheta^n(0)\right\|^{2}_{H^{4}(\mathbb R^{2})}+2\tilde C_8\int_0^T\left\|\vartheta^n(s)\right\|^2_{H^4(\mathbb R^2)}ds\\
     &\qquad+16\left(\tilde C_8\varepsilon\right)^2\left(\left\|\vartheta^{n-1}\right\|^2_{L^\infty(]0, T]; H^4(\mathbb R^2))}+\frac18\int_0^T\left\|\partial_y\vartheta^{n-1}(s)\right\|^2_{H^4(\mathbb R^2)}ds\right),
\end{align*}
due to the arbitrary of $T>0$, the above inequality also holds if we replace $T$ with $t$. Thus, together with 
\begin{align*}
    \left\|\vartheta^n(0)\right\|^{2}_{H^{4}(\mathbb R^{2})}=\left\|\langle \delta y\rangle^{-1}\zeta^n(0)\right\|^{2}_{H^{4}(\mathbb R^{2})}\le\left\|\zeta^n(0)\right\|^{2}_{H^{4}(\mathbb R^{2})}\le2\tilde B_0\varepsilon^2,
\end{align*}
we can get that for all $0<t\le T$
\begin{align*}
    &\left\|\vartheta^n(t)\right\|^{2}_{H^{4}(\mathbb R^{2})}+\frac{1}{8}\int_0^t\left\|\partial_y\vartheta^n(s)\right\|^2_{H^4(\mathbb R^2)}ds\le4\tilde B_0\varepsilon^2+4\tilde C_8\int_0^t\left\|\vartheta^n(s)\right\|^2_{H^4(\mathbb R^2)}ds\\
     &\qquad+32\left(\tilde C_8\varepsilon\right)^2\left(\left\|\vartheta^{n-1}\right\|^2_{L^\infty(]0, t]; H^4(\mathbb R^2))}+\frac18\int_0^t\left\|\partial_y\vartheta^{n-1}(s)\right\|^2_{H^4(\mathbb R^2)}ds\right).
\end{align*}
Applying the Gronwall's inequality to get
\begin{align*}
    &\left\|\vartheta^n(t)\right\|^{2}_{H^{4}(\mathbb R^{2})}\le16\varepsilon^2\tilde C_8e^{4\tilde C_8T}\left(\tilde B_0+8\left(\tilde C_8\right)^2\left(\left\|\vartheta^{n-1}\right\|^2_{L^\infty(]0, t]; H^4(\mathbb R^2))}+\frac18\int_0^t\left\|\partial_y\vartheta^{n-1}(s)\right\|^2_{H^4(\mathbb R^2)}ds\right)\right),
\end{align*}
this yields that for all $0<t\le T$
\begin{align*}
    &\left\|\vartheta^n(t)\right\|^{2}_{H^{4}(\mathbb R^{2})}+\frac{1}{8}\int_0^t\left\|\partial_y\vartheta^n(s)\right\|^2_{H^4(\mathbb R^2)}ds\\
    &\le4\varepsilon^2\left(4\tilde C_8Te^{4\tilde C_8T}+1\right)^2\left(\tilde B_0+8\left(\tilde C_8\right)^2\left(\left\|\vartheta^{n-1}\right\|^2_{L^\infty(]0, t]; H^4(\mathbb R^2))}+\frac18\int_0^t\left\|\partial_y\vartheta^{n-1}(s)\right\|^2_{H^4(\mathbb R^2)}ds\right)\right)\\
    &\le\frac12\left(\tilde B_0\varepsilon+\left\|\vartheta^{n-1}\right\|^2_{L^\infty(]0, t]; H^4(\mathbb R^2))}+\frac18\int_0^t\left\|\partial_y\vartheta^{n-1}(s)\right\|^2_{H^4(\mathbb R^2)}ds\right)\\
    &\le\frac12\left(\tilde B_0\varepsilon+\left\|\vartheta^{n-1}\right\|^2_{L^\infty(]0, T]; H^4(\mathbb R^2))}+\frac18\int_0^T\left\|\partial_y\vartheta^{n-1}(s)\right\|^2_{H^4(\mathbb R^2)}ds\right),
\end{align*}
if we choose $32\varepsilon\left(4\tilde C_8Te^{4\tilde C_8T}+1\right)^2\left(\tilde C_8\right)^2\le\frac12$. Then using the arbitrary of $T>0$, one can obtain that for all $0<t\le T$
\begin{align*}
    &\left\|\vartheta^n(t)\right\|^{2}_{L^\infty(]0, t]; H^4(\mathbb R^2))}+\frac{1}{8}\int_0^t\left\|\partial_y\vartheta^n(s)\right\|^2_{H^4(\mathbb R^2)}ds\\
    &\le\frac12\left(\tilde B_0\varepsilon+\left\|\vartheta^{n-1}\right\|^2_{L^\infty(]0, t]; H^4(\mathbb R^2))}+\frac18\int_0^t\left\|\partial_y\vartheta^{n-1}(s)\right\|^2_{H^4(\mathbb R^2)}ds\right).
\end{align*}
By taking the limit as $\delta\to0$ on the both side of above inequality, then one immediately has 
\begin{align*}
    &\left\|\zeta^n(t)\right\|^{2}_{L^\infty(]0, t]; H^4(\mathbb R^2))}+\frac{1}{8}\int_0^t\left\|\partial_y\zeta^n(s)\right\|^2_{H^4(\mathbb R^2)}ds\\
    &\le\frac12\left(\tilde B_0\varepsilon+\left\|\zeta^{n-1}\right\|^2_{L^\infty(]0, t]; H^4(\mathbb R^2))}+\frac18\int_0^t\left\|\partial_y\zeta^{n-1}(s)\right\|^2_{H^4(\mathbb R^2)}ds\right),
\end{align*}
from iterating, one can obtain that
\begin{align*}
    &\left\|\zeta^n(t)\right\|^{2}_{L^\infty(]0, t]; H^4(\mathbb R^2))}+\frac{1}{8}\int_0^t\left\|\partial_y\zeta^n(s)\right\|^2_{H^4(\mathbb R^2)}ds\\
    &\le\tilde B_0\varepsilon\sum_{0\le j\le n}2^{-j}+2^{-n}\left(\left\|\zeta^{0}\right\|^2_{L^\infty(]0, t]; H^4(\mathbb R^2))}+\frac18\int_0^t\left\|\partial_y\zeta^{0}(s)\right\|^2_{H^4(\mathbb R^2)}ds\right)\to 0, \ n\to\infty.
\end{align*}
Hence, we gain that $\left\{\zeta^n\right\}\subset H^{4}(\mathbb R^{2})$ is a Cauchy sequence, which provides that there exists a solution $v\in L^{\infty}(]0, T]; H^{4}(\mathbb R^{3}))$ to the Cauchy problem \eqref{2-1} such that 
$$\left\|v(t)\right\|^{2}_{H^{4}(\mathbb R^{2})}+\frac14\int_{0}^{t}\left\| \partial_y v(s)\right\|_{H^4(\mathbb R^{2})}^2ds\le \tilde B_0 \varepsilon^{2},\ \ \ \ \ 0<t\le T.$$
\end{proof}

\section{Energy Estimate for Derivations}\label{section4}

This section gives the energy estimate for high-order derivatives, as a  preparation for proving Theorem \ref{Theorem}. In order to state our result, we define the following auxiliary vector fields $\mathcal{H}_\eta$, 
\begin{equation}\label{auxiliary vector fields}
\mathcal{H}_\eta = \frac{t^{\eta + 1}}{\eta + 1} \partial_x + t^\eta \partial_y,
\end{equation}
where $\eta > 1$. To simplify the notation, we  show the operator $\mathcal{H} = \mathcal{H}_\eta$ has the following two properties. 
\begin{lemma}(\cite{u-1})\label{uu-1}
For any suitable functions $f$ and $g$, we have following Leibiniz-type formula
\begin{align}\label{Leibiniz-type formula}
     \mathcal{H}^{k}\left(f \cdot g\right)=\sum_{j=0}^{k} \binom{k}{j} \left( \mathcal{H}^{j} f \right) \left( \mathcal{H}^{k-j} g \right), \quad \forall \ k\in\mathbb N.
\end{align}
\end{lemma}
Hereafter, the notation $[T_{1}, T_{2}]=T_{1}T_{2}-T_{2}T_{1}$ is the commutator of the operators $T_1, T_2$.
\begin{lemma}(\cite{ref1})\label{lemma2.1}
Assume the operator $H$ defined as \eqref{auxiliary vector fields}, then for any $k\in\mathbb N_{+}$
\begin{equation}\label{1-5}
\left[ \partial_t + y \partial_x, \mathcal{H}^k \right] = k \delta t^{\delta - 1} \partial_y \mathcal{H}^{k - 1}.
\end{equation}
\end{lemma}
To establish the energy estimate for high-order derivatives stated in Proposition \ref{Proposition 3.1.}, we first present the following two lemmas.
\begin{lemma}(\cite{K-P})\label{KP}(Kato-Ponce type commutator estimate)
    For any $s\in\mathbb{R}$  and $1<p<\infty$, denote $\Lambda=\sqrt{1-\Delta}$, we have
    \begin{align}\label{Kato-Ponce-c}
        \left\|\Lambda^s(fg)-f \Lambda^s g\right\|_{L^p}\lesssim\|\nabla f\|_{L^\infty}\|\Lambda^{s-1} g\|_{L^p}+\|\Lambda^s f\|_{L^p}\|g\|_{L^\infty}.
    \end{align}
    %and
    %\begin{align}\label{Kato-Ponce}
    %\left\|\Lambda^s(fg)\right\|_{L^p}\lesssim\left\|\Lambda^sf\right\|_{L^{p_1}}\left\|g\right\|_{L^{p_2}}+\left\|\Lambda^sg\right\|_{L^{p_3}}\left\|f\right\|_{L^{p_4}},
    %\end{align}
    %where $\frac{1}{p}=\frac{1}{p_1}+\frac{1}{p_2}=\frac{1}{p_3}+\frac{1}{p_4}$, $1<p_1, p_3<\infty$.
\end{lemma}

\begin{lemma}\label{lem4.4}
    For any $0<\sigma<1$ and $s \in \mathbb{R}$, define $\Delta_\sigma=1-\sigma\Delta_{x, y}$, then the operator $\Delta_\sigma^{-1}(\sigma\partial_x\partial_y)$ and $\sigma\Lambda^2\Delta_\sigma^{-1}$ are bounded on $H^s(\mathbb R^2)$.
\end{lemma}
\begin{proof}
    Let $f\in H^s(\mathbb R^2)$, consider $\Delta_\sigma^{-1}(\sigma\partial_x\partial_y f)$ and $\sigma\Lambda^2\Delta_\sigma^{-1}f$ on $H^s(\mathbb R^2)$ separately. Then 
    \begin{align*}
        \|\Delta_\sigma^{-1}(\sigma\partial_x\partial_y f)\|^2_{H^s(\mathbb R^2)}&=\int_{\mathbb R^2}(1+|\xi|^2)^s\left|\mathcal{F}[\Delta_\sigma^{-1}(\sigma\partial_x\partial_y f)](\xi)\right|^2d\xi\\
        &=\int_{\mathbb R^2}(1+|\xi|^2)^s\left|(1+\sigma|\xi|^2)^{-1}\sigma\xi_1\xi_2\mathcal{F}[f](\xi)\right|^2d\xi\\
        &\le\int_{\mathbb R^2}(1+|\xi|^2)^s\left|(1+\sigma|\xi|^2)^{-1}\sigma|\xi|^2\mathcal{F}[f](\xi)\right|^2d\xi\le\|f\|^2_{H^s(\mathbb R^2)}.
    \end{align*}
    and 
    \begin{align*}
        \|\sigma\Lambda^2\Delta_\sigma^{-1}f\|^2_{H^s(\mathbb R^2)}&=\int_{\mathbb R^2}(1+|\xi|^2)^s\left|\mathcal{F}[\sigma\Lambda^2\Delta_\sigma^{-1}f](\xi)\right|^2d\xi\\
        &=\int_{\mathbb R^2}(1+|\xi|^2)^s\left|\sigma(1+|\xi|^2)(1+\sigma|\xi|^2)^{-1}\mathcal{F}[f](\xi)\right|^2d\xi\le\|f\|^2_{H^s(\mathbb R^2)}.
    \end{align*}
\end{proof}

We then derive the energy estimates for high-order derivatives, where all constants implied in the notation $\lesssim$ are uniformly independent of the induction index.

\begin{proposition}\label{Proposition 3.1.}
Let $T > 0$, $v_0 \in H^4(\mathbb R^{2})$, and there exists a sufficiently small constant $\varepsilon > 0$ such that
\begin{align*}%\label{initial small} 
     \left\|v_{0}\right\|_{H^{4}(\mathbb R^{2})}\le\varepsilon.
\end{align*}  
Let $v$ be the solution of the Cauchy problem \eqref{2-1}, then there exists a constant $A > 1$ which independent of $k$, such that for any $t \in \left] 0, T \right]$,  $k \in \mathbb{N}$, 
\begin{align}\label{H-k}
\left\| \mathcal{H}^k v(t) \right\|^2_{H^4(\mathbb{R}^2)} + \frac14 \int_0^t \left\| \partial_y \mathcal{H}^k v(s) \right\|^2_{H^4(\mathbb{R}^2)} ds \leq \left( \frac{A^{k -1} k!}{(k+1)^2} \right)^2.
\end{align}
\end{proposition}
\begin{proof}
We prove this proposition by induction on the index $k$. For $k=0$, it has been already proved in \eqref{2-0-1}. Assume that $k\ge1$ and \eqref{H-k} holds for all $0<t\le T$ and $q\in\mathbb N$ with $0\le q\le k-1$
\begin{align}\label{hypothesis}
\left\| \mathcal{H}^q v(t) \right\|^2_{H^4(\mathbb{R}^2)} +  \frac14 \int_0^t \left\| \partial_y \mathcal{H}^q v(s) \right\|^2_{H^4(\mathbb{R}^2)} ds \leq \left( \frac{A^{q -1} q!}{(q+1)^2} \right)^2.
\end{align}
Now, we shall prove \eqref{H-k} for $q = k$. Using the equation  the commutator \eqref{1-5}, we have
\begin{align*}%\label{4-1}
    \partial_t \mathcal{H}^{k} v + y\partial_x \mathcal{H}^{k} v - \mathcal{H}^{k} \left((1 + v)^2 \partial_{y}^2 v\right)=  k\delta t^{\delta-1}\partial_{y}\mathcal{H}^{k-1}v.
\end{align*}

However, since integration by parts does not work when taking scalar product in $H^4(\mathbb{R}^2)$ with respect to $\mathcal{H}^k v$, we would use the following strategy: We first define 
$$
v^k_\sigma=\triangle^{-1}_\sigma \mathcal{H}^k v\in L^\infty([0, T]; H^5(\mathbb{R}^2)), \ \ \ \ 
\partial_y v^k_\sigma
\in  L^2([0, T]; H^5(\mathbb{R}^2)),
$$
where
$$
\triangle_\sigma=1-\sigma\triangle_{x, y}, \quad 0<\sigma\ll 1.
$$
Then we would take the scalar product in $H^4(\mathbb{R}^2)$ with $ \triangle^{-1}_\sigma v^k_\sigma$, 
\begin{equation}\label{inden}
\begin{aligned}
\left( \triangle^{-1}_\sigma (\partial_t + y\partial_x) \mathcal{H}^{k} v,   v^k_\sigma\right)_{H^4(\mathbb{R}^2)} & - 
\left( \triangle^{-1}_\sigma \mathcal{H}^{k} \left((1 + v)^2 \partial_{y}^2 v\right) ,   v^k_\sigma\right)_{H^4(\mathbb{R}^2)}\\
& = \left(\triangle^{-1}_\sigma k\delta t^{\delta-1}\partial_{y}\mathcal{H}^{k-1}v,   v^k_\sigma\right)_{H^4(\mathbb{R}^2)}.
\end{aligned}
\end{equation}
{\em And we would give an estimate of the above equality, and then let $\sigma \rightarrow 0$ to obtain the desired inequality.}

We now give an estimate of the identity \eqref{inden}. First, it is clear that
\begin{align*}
&\triangle^{-1}_\sigma (\partial_t + y\partial_x) \mathcal{H}^{k} v= \triangle^{-1}_\sigma (\partial_t + y\partial_x) \triangle_\sigma\triangle^{-1}_\sigma \mathcal{H}^{k} v\\
& \qquad = (\partial_t + y\partial_x) \triangle^{-1}_\sigma \mathcal{H}^{k} v +
	\triangle^{-1}_\sigma [(\partial_t + y\partial_x),  \triangle_\sigma]\triangle^{-1}_\sigma \mathcal{H}^{k} v\\
	& \qquad = (\partial_t + y\partial_x) \triangle^{-1}_\sigma \mathcal{H}^{k} v +
	\triangle^{-1}_\sigma  (\sigma \partial_x\partial_y)\triangle^{-1}_\sigma \mathcal{H}^{k} v\\
	& \qquad = (\partial_t + y\partial_x) v^k_\sigma+T_\sigma v^k_\sigma ,
\end{align*}
where $T_\sigma=\triangle^{-1}_\sigma  (\sigma \partial_x\partial_y)$ is a bounded operator on Sobolev space, using the Cauchy-Schwarz inequality
\begin{align*}
    \left|(T_\sigma v^k_\sigma, v^k_\sigma)_{H^4(\mathbb R^2)}\right|\lesssim\left\|v^k_\sigma\right\|^2_{H^4(\mathbb R^2)}.
\end{align*}
As demonstrated in the proof of Proposition \ref{existence111}, notice that $\Delta_\sigma$ is a Fourier multipiler with the symbol $(1+\sigma|\xi|^2)^{-1}$, we can obtain the following estimate: 
$$
\left| \left( y\partial_x v_\sigma^k, v_\sigma^k \right)_{H^4(\mathbb{R}^2)} \right| \lesssim \left\| v_\sigma^k \right\|_{H^4(\mathbb{R}^2)}.
$$

We observe a few simple facts about $\left(\Delta_\sigma^{-1}\mathcal{H}^{k} \left((1 + v)^2 \partial_{y}^2 v\right), v^k_\sigma\right)_{H^4(\mathbb R^2)}$, by using the Leibniz-type formula \eqref{Leibiniz-type formula}, it follows that 
\begin{align*}
&\left(\Delta_\sigma^{-1}\mathcal{H}^{k} \left((1 + v)^2 \partial_{y}^2 v\right), v^k_\sigma\right)_{H^4(\mathbb R^2)}\\
&=\left(\Delta_\sigma^{-1}\partial_y \mathcal{H}^{k} \left((1 + v)^2 \partial_{y} v\right), v^k_\sigma\right)_{H^4(\mathbb R^2)}-\left(\Delta_\sigma^{-1}\mathcal{H}^{k} \left(\partial_y(1 + v)^2 \partial_{y} v\right), v^k_\sigma\right)_{H^4(\mathbb R^2)}\\
&=-\left(\Delta_\sigma^{-1}\mathcal{H}^{k} \left((1 + v)^2 \partial_{y} v\right), \partial_yv^k_\sigma\right)_{H^4(\mathbb R^2)}-\left(\Delta_\sigma^{-1}\mathcal{H}^{k} \left(\partial_y(1 + v)^2 \partial_{y} v\right), v^k_\sigma\right)_{H^4(\mathbb R^2)}\\
&=-\sum_{j=0}^{k} \binom{k}{j} \left(\Delta_\sigma^{-1}\left[( \mathcal{H}^j (1 + v)^2) ( \mathcal{H}^{k-j}  \partial_{y} v)\right], \partial_yv^k_\sigma\right)_{H^4(\mathbb R^2)}\\%-\left(\Delta_\sigma^{-1}\left[ (1 + v)^2  H^{k}  \partial_{y} v\right], \partial_yv^k_\sigma\right)_{H^4(\mathbb R^2)}\\
&\qquad-\sum_{j=0}^{k} \binom{k}{j}\left(\Delta_\sigma^{-1}\left[\mathcal{H}^{j} \left(\partial_y(1 + v)^2\right)\left( \mathcal{H}^{k-j}\partial_{y} v\right)\right], v^k_\sigma\right)_{H^4(\mathbb R^2)}\\%-\left(\Delta_\sigma^{-1}\partial_y(1 + v)^2 H^{k}\partial_{y} v, v^k_\sigma\right)_{H^4(\mathbb R^2)}\\
&=-\sum_{j=0}^{k}S_{1, j}-\sum_{j=0}^{k}S_{2, j}.
\end{align*}
We shall compute both terms separately. For the $S_{1, j}$ with $j=0, 1, 2, \cdots, k$, we can write it as
\begin{align*}
    S_{1, j}&=\left(\Delta_\sigma^{-1}\left[( \mathcal{H}^j (1 + v)^2)\Delta_\sigma \Delta_\sigma^{-1} \mathcal{H}^{k-j}  \partial_{y} v\right], \partial_yv^k_\sigma\right)_{H^4(\mathbb R^2)}\\
    &=\left(\Delta_\sigma^{-1}\left[( \mathcal{H}^j (1 + v)^2)\Delta_\sigma \partial_y v_\sigma^{k-j}\right], \partial_yv^k_\sigma\right)_{H^4(\mathbb R^2)}\\
    &=\left( \mathcal{H}^j (1 + v)^2 \partial_y v_\sigma^{k-j}, \partial_yv^k_\sigma\right)_{H^4(\mathbb R^2)}+\left(\Delta_\sigma^{-1}\left[ \mathcal{H}^j (1 + v)^2, \Delta_\sigma\right] \partial_y v_\sigma^{k-j}, \partial_yv^k_\sigma\right)_{H^4(\mathbb R^2)}\\
    &=S_{1, j, 1}+S_{1, j, 2}.
\end{align*}
If $1\le j\le k$, then from the Cauchy-Schwarz inequality
\begin{align*}
    \left|S_{1, j, 1}\right|&\le\left\| \mathcal{H}^j (1 + v)^2 \partial_y v_\sigma^{k-j}\right\|_{H^4(\mathbb R^2)} \left\|\partial_yv^k_\sigma\right\|_{H^4(\mathbb R^2)}\\
    &\lesssim\left\| \mathcal{H}^j (1 + v)^2 \right\|_{H^4(\mathbb R^2)} \left\|\partial_y v_\sigma^{k-j}\right\|_{H^4(\mathbb R^2)} \left\|\partial_yv^k_\sigma\right\|_{H^4(\mathbb R^2)},
\end{align*}
here we use $\|fg\|_{H^s}\le C_s\|f\|_{H^s}\|g\|_{H^s}$ for any $f, g\in H^s$.

We observe that the term $S_{1, 0, 1}$ can be written it as
\begin{align*}
    S_{1, 0, 1}&=\left( (1 + v)^2 \partial_y v_\sigma^{k}, \partial_yv^k_\sigma\right)_{H^4(\mathbb R^2)}=\left(\Lambda^4\left((1 + v)^2 \partial_y v_\sigma^{k}\right), \Lambda^4\partial_yv^k_\sigma\right)_{L^2(\mathbb R^2)}\\
    &=\left((1 + v)^2 \Lambda^4\partial_y v_\sigma^{k}, \Lambda^4\partial_yv^k_\sigma\right)_{L^2(\mathbb R^2)}+\left(\left[\Lambda^4, (1 + v)^2\right] \partial_y v_\sigma^{k}, \Lambda^4\partial_y v^k_\sigma\right)_{L^2(\mathbb R^2)}.
\end{align*}
Since $\|v\|_{L^\infty([0, T]; H^4(\mathbb R^2))}\le\varepsilon\ll1$, then 
\begin{align*}
    \left((1 + v)^2 \Lambda^4\partial_y v_\sigma^{k}, \Lambda^4\partial_yv^k_\sigma\right)_{L^2(\mathbb R^2)}\ge\frac12\|\Lambda^4\partial_y v_\sigma^{k}\|^2_{L^2(\mathbb R^2)}=\frac12\|\partial_y v_\sigma^{k}\|^2_{H^4(\mathbb R^2)}.
\end{align*}
We apply the the second term via the commutator estimate \eqref{Kato-Ponce-c}, 
\begin{align*}
    &\left|\left(\left[\Lambda^4, (1 + v)^2\right] \partial_y v_\sigma^{k}, \Lambda^4\partial_y v^k_\sigma\right)_{L^2(\mathbb R^2)}\right|\\
    &\le\left\|\left[\Lambda^4, (1 + v)^2\right] \partial_y v_\sigma^{k}\right\|_{L^2(\mathbb R^4)}\|\Lambda^4\partial_y v^k_\sigma\|_{L^2(\mathbb R^2)}\\
    &\lesssim\left(\|\nabla_{x, y}(1+v)^2\|_{L^\infty(\mathbb R^2)}\|\Lambda^3 \partial_y v^k_\sigma\|_{L^2(\mathbb R^2)}+\|\Lambda ^4(1+v)^2\|_{L^2(\mathbb R^2)}\|\partial_y v^k_\sigma\|_{L^\infty(\mathbb R^2)}\right)\|\Lambda^4\partial_y v^k_\sigma\|_{L^2(\mathbb R^2)}\\
    &\lesssim\left(\|\nabla_{x, y}(1+v)^2\|_{L^\infty(\mathbb R^2)}\|\Lambda^3 \partial_y v^k_\sigma\|_{L^2(\mathbb R^2)}+\|\Lambda ^4(1+v)\|^2_{L^2(\mathbb R^2)}\|\partial_y v^k_\sigma\|_{L^\infty(\mathbb R^2)}\right)\|\Lambda^4\partial_y v^k_\sigma\|_{L^2(\mathbb R^2)},
\end{align*}
where $\Lambda=(1-\Delta_{x, y})^{\frac12}$. Using Sobolev embedding $\|\cdot\|_{L^\infty(\mathbb R^2)}\lesssim\|\cdot\|_{H^{1+}(\mathbb R^2)}$, one has
\begin{align*}
    \|\nabla_{x, y}(1+v)^2\|_{L^\infty(\mathbb R^2)}\lesssim\|(1+v)^2\|_{H^3(\mathbb R^2)}\lesssim\|v\|^2_{H^4(\mathbb R^2)}, \qquad \|\partial_y v^k_\sigma\|_{L^\infty(\mathbb R^2)}\lesssim\|\partial_y v^k_\sigma\|_{H^4(\mathbb R^2)},
\end{align*}
this implies 
\begin{align*}
    &\left|\left(\left[\Lambda^4, (1 + v)^2\right] \partial_y v_\sigma^{k}, \Lambda^4\partial_y v^k_\sigma\right)_{L^2(\mathbb R^2)}\right| \lesssim \|v\|^2_{H^4(\mathbb R^2)}\|\partial_y v^k_\sigma\|^2_{H^4(\mathbb R^2)}.
\end{align*}
Hence, choosing suitable $\varepsilon$, we can deduce that 
\begin{align*}
    S_{1, 0, 1}\ge \frac14\|\partial_y v^k_\sigma\|_{H^4(\mathbb R^2)}.
\end{align*}
Noting that 
\begin{align*}
    \left[ \mathcal{H}^j (1 + v)^2, \Delta_\sigma\right]=-\sigma\left[ \mathcal{H}^j (1 + v)^2, \Delta_{x, y}\right]=-\sigma\left(\Delta_{x, y}\mathcal{H}^j (1 + v)^2+\nabla_{x, y} \mathcal{H}^j (1 + v)^2 \nabla_{x, y}\right),
\end{align*}
then $S_{1, j, 2}$ can be written as
\begin{align*}
    S_{1, j, 2}&=-\sigma\left(\Delta_\sigma^{-1}\left(\Delta_{x, y}\mathcal{H}^j (1 + v)^2 \partial_y v_\sigma^{k-j}\right), \partial_yv^k_\sigma\right)_{H^4(\mathbb R^2)}\\
    &\quad-\sigma\left(\Delta_\sigma^{-1}\left( \nabla_{x, y} \mathcal{H}^j (1 + v)^2 \nabla_{x, y}\partial_y v_\sigma^{k-j}\right), \partial_yv^k_\sigma\right)_{H^4(\mathbb R^2)}\\
    &=-\sigma\left(\Lambda^4\Delta_\sigma^{-1}\left(\Delta_{x, y}\mathcal{H}^j (1 + v)^2 \partial_y v_\sigma^{k-j}\right), \Lambda^4\partial_yv^k_\sigma\right)_{L^2(\mathbb R^2)}\\
    &\quad-\sigma\left(\Lambda^4\Delta_\sigma^{-1}\left( \nabla_{x, y} \mathcal{H}^j (1 + v)^2 \nabla_{x, y}\partial_y v_\sigma^{k-j}\right), \Lambda^4\partial_yv^k_\sigma\right)_{L^2(\mathbb R^2)}.
\end{align*}
In view of the correspondence for Lemma \ref{lem4.4} that $\sigma\Lambda^2\Delta_\sigma^{-1}$ is a bounded operator on $L^2(\mathbb R^2)$, by applying the Cauchy-Schwarz inequality,
\begin{align*}
    &\left|-\sigma\left(\Lambda^4\Delta_\sigma^{-1}\left(\Delta_{x, y}\mathcal{H}^j (1 + v)^2 \partial_y v_\sigma^{k-j}\right), \Lambda^4\partial_yv^k_\sigma\right)_{L^2(\mathbb R^2)}\right|\\
    &\le\left\|\Lambda^2\left(\Delta_{x, y}\mathcal{H}^j (1 + v)^2 \partial_y v_\sigma^{k-j}\right)\right\|_{L^2(\mathbb R^2)}\left\|\Lambda^4\partial_yv^k_\sigma\right\|_{L^2(\mathbb R^2)}\\
    &\lesssim\left\|\Lambda^2\Delta_{x, y}\mathcal{H}^j (1 + v)^2 \right\|_{L^2(\mathbb R^2)}\left\|\Lambda^2 \partial_y v_\sigma^{k-j}\right\|_{L^2(\mathbb R^2)}\left\|\Lambda^4\partial_yv^k_\sigma\right\|_{L^2(\mathbb R^2)},
\end{align*}
and
\begin{align*}
    &\left|-\sigma\left(\Lambda^4\Delta_\sigma^{-1}\left( \nabla_{x, y} \mathcal{H}^j (1 + v)^2 \nabla_{x, y}\partial_y v_\sigma^{k-j}\right), \Lambda^4\partial_yv^k_\sigma\right)_{L^2(\mathbb R^2)}\right|\\
    &\le\left\|\Lambda^2\left( \nabla_{x, y} \mathcal{H}^j (1 + v)^2 \nabla_{x, y}\partial_y v_\sigma^{k-j}\right)\right\|_{L^2(\mathbb R^2)}
    \left\| \Lambda^4\partial_yv^k_\sigma\right\|_{L^2(\mathbb R^2)}\\
    &\lesssim\left\|\Lambda^2\nabla_{x, y} \mathcal{H}^j (1 + v)^2 \right\|_{L^2(\mathbb R^2)}\left\|\Lambda^2 \nabla_{x, y}\partial_y v_\sigma^{k-j}\right\|_{L^2(\mathbb R^2)}
    \left\| \Lambda^4\partial_yv^k_\sigma\right\|_{L^2(\mathbb R^2)},
\end{align*}
then using Sobolev embedding, in this case of $1\le j\le k$, we have
\begin{align*}
    \left|S_{1, j, 2}\right|&\le\left|-\sigma\left(\Lambda^4\Delta_\sigma^{-1}\left(\Delta_{x, y}\mathcal{H}^j (1 + v)^2 \partial_y v_\sigma^{k-j}\right), \Lambda^4\partial_yv^k_\sigma\right)_{L^2(\mathbb R^2)}\right|\\
    &\quad+\left|-\sigma\left(\Lambda^4\Delta_\sigma^{-1}\left( \nabla_{x, y} \mathcal{H}^j (1 + v)^2 \nabla_{x, y}\partial_y v_\sigma^{k-j}\right), \Lambda^4\partial_yv^k_\sigma\right)_{L^2(\mathbb R^2)}\right|\\
    &\lesssim\left\|\mathcal{H}^j(1+v)^2\right\|_{H^4(\mathbb R^2)}\left\|\partial_yv^{k-j}_\sigma\right\|_{H^4(\mathbb R^2)}\left\|\partial_yv^k_\sigma\right\|_{H^4(\mathbb R^2)},
\end{align*}
and 
\begin{align*}
    \left|S_{1, 0, 2}\right|&\le\left|-\sigma\left(\Lambda^4\Delta_\sigma^{-1}\left(\Delta_{x, y}(1 + v)^2 \partial_y v_\sigma^{k}\right), \Lambda^4\partial_yv^k_\sigma\right)_{L^2(\mathbb R^2)}\right|\\
    &\quad+\left|-\sigma\left(\Lambda^4\Delta_\sigma^{-1}\left( \nabla_{x, y} (1 + v)^2 \nabla_{x, y}\partial_y v_\sigma^{k}\right), \Lambda^4\partial_yv^k_\sigma\right)_{L^2(\mathbb R^2)}\right|\\
    &\lesssim\left\|v\right\|^2_{H^4(\mathbb R^2)}\left\|\partial_yv^k_\sigma\right\|^2_{H^4(\mathbb R^2)},
\end{align*}

We now estimate $S_{2, j}$ for $j=0, 1, \cdots, k$. We can rewrite it as 
\begin{align*}
    S_{2, j}&=\left(\Delta_\sigma^{-1}\left[\mathcal{H}^{j} \left(\partial_y(1 + v)^2\right)\Delta_\sigma\Delta_\sigma^{-1}\left( \mathcal{H}^{k-j}\partial_{y} v\right)\right], v^k_\sigma\right)_{H^4(\mathbb R^2)}\\
    &=\left(\Delta_\sigma^{-1}\left[\mathcal{H}^{j} \left(\partial_y(1 + v)^2\right)\Delta_\sigma\partial_y v^{k-j}_\sigma\right], v^k_\sigma\right)_{H^4(\mathbb R^2)}\\
    &=\left(\mathcal{H}^{j} \left(\partial_y(1 + v)^2\right)\partial_y v^{k-j}_\sigma, v^k_\sigma\right)_{H^4(\mathbb R^2)}+\left(\Delta_\sigma^{-1}\left[\mathcal{H}^{j} \left(\partial_y(1 + v)^2\right), \Delta_\sigma\right]\partial_y v^{k-j}_\sigma, v^k_\sigma\right)_{H^4(\mathbb R^2)}.
\end{align*}
Using the Cauchy-Schwarz inequality, we can get that 
\begin{align*}
    &\left|\left(\mathcal{H}^{j} \left(\partial_y(1 + v)^2\right)\partial_y v^{k-j}_\sigma, v^k_\sigma\right)_{H^4(\mathbb R^2)}\right|\\
    &\le\left\|\Lambda^4\left(\mathcal{H}^{j} \left(\partial_y(1 + v)^2\right)\partial_y v^{k-j}_\sigma\right)\right\|_{L^2(\mathbb R^2)}\left\|v^k_\sigma\right\|_{H^4(\mathbb R^2)}\\
    &\lesssim\left\|\mathcal{H}^{j} \left(\partial_y(1 + v)^2\right)\right\|_{H^4(\mathbb R^2)}\left\|\partial_y v^{k-j}_\sigma\right\|_{H^4(\mathbb R^2)}\left\|v^k_\sigma\right\|_{H^4(\mathbb R^2)}
\end{align*}
For the second term in $S_{2, j}$, noting that 
\begin{align*}
    \left[ \mathcal{H}^j \partial_y(1 + v)^2, \Delta_\sigma\right]=-\sigma\left[ \mathcal{H}^j \partial_y(1 + v)^2, \Delta_{x, y}\right]=-\sigma\left(\Delta_{x, y}\mathcal{H}^j \partial_y(1 + v)^2+\nabla_{x, y} \mathcal{H}^j \partial_y(1 + v)^2 \nabla_{x, y}\right),
\end{align*}
then we have
\begin{align*}
    &\left(\Delta_\sigma^{-1}\left[\mathcal{H}^{j} \left(\partial_y(1 + v)^2\right), \Delta_\sigma\right]\partial_y v^{k-j}_\sigma, v^k_\sigma\right)_{H^4(\mathbb R^2)}\\
    &=-\sigma\left(\Delta_\sigma^{-1}\left[\Delta_{x, y}\mathcal{H}^j \partial_y(1 + v)^2\partial_y v^{k-j}_\sigma\right], v^k_\sigma\right)_{H^4(\mathbb R^2)}\\
    &\quad-\sigma\left(\Delta_\sigma^{-1}\left[\nabla_{x, y} \mathcal{H}^j \partial_y(1 + v)^2 \nabla_{x, y}\partial_y v^{k-j}_\sigma\right], v^k_\sigma\right)_{H^4(\mathbb R^2)},
\end{align*}
as discussed in $S_{1, j, 2}$, we can obtain that 
\begin{align*}
    &\left|\left(\Delta_\sigma^{-1}\left[\mathcal{H}^{j} \left(\partial_y(1 + v)^2\right), \Delta_\sigma\right]\partial_y v^{k-j}_\sigma, v^k_\sigma\right)_{H^4(\mathbb R^2)}\right|\\
    &\le\left|-\sigma\left(\Lambda^4\Delta_\sigma^{-1}\left(\Delta_{x, y}\mathcal{H}^j \partial_y(1 + v)^2 \partial_y v_\sigma^{k-j}\right), \Lambda^4v^k_\sigma\right)_{L^2(\mathbb R^2)}\right|\\
    &\quad+\left|-\sigma\left(\Lambda^4\Delta_\sigma^{-1}\left( \nabla_{x, y} \mathcal{H}^j \partial_y(1 + v)^2 \nabla_{x, y}\partial_y v_\sigma^{k-j}\right), \Lambda^4v^k_\sigma\right)_{L^2(\mathbb R^2)}\right|\\
    &\lesssim\left\|\mathcal{H}^j\partial_y(1+v)^2\right\|_{H^4(\mathbb R^2)}\left\|\partial_yv^{k-j}_\sigma\right\|_{H^4(\mathbb R^2)}\left\|v^k_\sigma\right\|_{H^4(\mathbb R^2)}.
\end{align*}
As for the right term of the equality \eqref{inden}, by using the Cauchy-Schwarz inequality, we compute directly as follows: 
$$
\left| \left( \Delta_\sigma^{-1} k \eta t^{\eta - 1} \partial_y \mathcal{H}^{k - 1} v, v_\sigma^k \right)_{H^4(\mathbb{R}^2)} \right| \leq 2 \left( k\eta t^{\eta - 1} \right)^2 \left\| \partial_y v_\sigma^{k - 1} \right\|_{H^4(\mathbb{R}^2)}^2 + \frac18 \left\| v_\sigma^k \right\|_{H^4(\mathbb{R}^2)}^2.
$$

Combining these results, we have 
\begin{align*}
   &\frac12\frac{d}{dt}\|v^k_\sigma(t)\|^2_{H^4(\mathbb R^2)}+\frac14\|\partial_y v^k_\sigma\|_{H^4(\mathbb R^2)}
   \le B_1\left\|v_\sigma^k\right\|^2_{H^4(\mathbb R^2)}\left\|\partial_yv^k_\sigma\right\|^2_{H^4(\mathbb R^2)}+B_1\left\|v^k_\sigma\right\|^2_{H^4(\mathbb R^2)}\\
   &\qquad\qquad+B_1\sum_{j=1}^k\left\| \mathcal{H}^j (1 + v)^2 \right\|_{H^4(\mathbb R^2)} \left\|\partial_y v_\sigma^{k-j}\right\|_{H^4(\mathbb R^2)} \left\|\partial_yv^k_\sigma\right\|_{H^4(\mathbb R^2)}\\
   &\qquad\qquad+B_1\sum_{j=0}^k\left\| \mathcal{H}^j \partial_y(1 + v)^2 \right\|_{H^4(\mathbb R^2)} \left\|\partial_y v_\sigma^{k-j}\right\|_{H^4(\mathbb R^2)} \left\|v^k_\sigma\right\|_{H^4(\mathbb R^2)} \\
   &\qquad\qquad+2(k\eta t^{\eta - 1}) \left\| \partial_y v_\sigma^{k - 1} \right\|_{H^4(\mathbb{R}^2)}^2,
\end{align*}
where $B_1>0$ is a constant, independent of $k$. Taking $\sigma\to0$,
\begin{align*}
   &\frac12\frac{d}{dt}\|\mathcal{H}^k v(t)\|^2_{H^4(\mathbb R^2)}+\frac14\|\partial_y \mathcal{H}^k v\|_{H^4(\mathbb R^2)}
   \le B_1\left\|v\right\|^2_{H^4(\mathbb R^2)}\left\|\partial_y \mathcal{H}^k v\right\|^2_{H^4(\mathbb R^2)}+B_1\left\|\mathcal{H}^k v\right\|^2_{H^4(\mathbb R^2)}\\
   &\qquad\qquad+B_1\sum_{j=1}^k\left\| \mathcal{H}^j (1 + v)^2 \right\|_{H^4(\mathbb R^2)} \left\|\partial_y \mathcal{H}^{k-j}v\right\|_{H^4(\mathbb R^2)} \left\|\partial_y \mathcal{H}^k v\right\|_{H^4(\mathbb R^2)}\\
   &\qquad\qquad+B_1\sum_{j=0}^k\left\| \mathcal{H}^j \partial_y(1 + v)^2 \right\|_{H^4(\mathbb R^2)} \left\|\partial_y \mathcal{H}^{k-j}v\right\|_{H^4(\mathbb R^2)} \left\|\mathcal{H}^k v\right\|_{H^4(\mathbb R^2)} \\
   &\qquad\qquad+2(k\eta t^{\eta - 1})^2 \left\| \partial_y \mathcal{H}^{k - 1} v \right\|_{H^4(\mathbb{R}^2)}^2.
\end{align*}
For any $0<t\le T$, integrating from 0 to $t$
\begin{align*}
   &\frac12\|\mathcal{H}^k v(t)\|^2_{H^4(\mathbb R^2)}+\frac14\int_0^t\|\partial_y \mathcal{H}^k v(s)\|_{H^4(\mathbb R^2)}ds\\
   &\le\frac12\|\mathcal{H}^k v(t)\|^2_{H^4(\mathbb R^2)}\Big|_{t=0} + B_1\int_0^t\left\|\mathcal{H}^k v(s)\right\|^2_{H^4(\mathbb R^2)}ds + B_1\int_0^t\left\|v(s)\right\|^2_{H^4(\mathbb R^2)}\left\|\partial_y \mathcal{H}^k v(s)\right\|^2_{H^4(\mathbb R^2)}ds\\
   &\quad+B_1\sum_{j=1}^k\int_0^t\left\| \mathcal{H}^j (1 + v)^2 (s)\right\|_{H^4(\mathbb R^2)} \left\|\partial_y \mathcal{H}^{k-j}v(s)\right\|_{H^4(\mathbb R^2)} \left\|\partial_y \mathcal{H}^k v(s)\right\|_{H^4(\mathbb R^2)}ds\\
   &\quad+B_1\sum_{j=0}^k\int_0^t\left\| \mathcal{H}^j \partial_y(1 + v)^2(s) \right\|_{H^4(\mathbb R^2)} \left\|\partial_y \mathcal{H}^{k-j}v(s)\right\|_{H^4(\mathbb R^2)} \left\|\mathcal{H}^k v(s)\right\|_{H^4(\mathbb R^2)}ds\\
   &\quad+2\int_0^t (k\eta t^{\eta - 1})^2 \left\| \partial_y \mathcal{H}^{k - 1} v(s) \right\|_{H^4(\mathbb{R}^2)}^2 ds \\
   &=B_1\int_0^t\left\|\mathcal{H}^k v(s)\right\|^2_{H^4(\mathbb R^2)}ds+Q_1+Q_2+Q_3+Q_4.
\end{align*}
where
\begin{align*}
\|\mathcal{H}^k v(t)\|^2_{H^4(\mathbb R^2)}\Big|_{t=0}=0.
\end{align*}
Using Proposition \ref{existence111}, the term $Q_1$ can be bounded as
\begin{align*}
    Q_1&\le B_1\left\|v\right\|^2_{L^\infty([0, T]; H^4(\mathbb R^2))}\int_0^t\left\|\partial_y \mathcal{H}^k v(s)\right\|^2_{H^4(\mathbb R^2)}ds\\
    &\le B_2\varepsilon\int_0^t\left\|\partial_y \mathcal{H}^k v(s)\right\|^2_{H^4(\mathbb R^2)}ds.
\end{align*}
For $Q_2$, we can directly compute
\begin{align*}
    Q_2&\le B_1\sum_{j=1}^k\int_0^t\left\| \mathcal{H}^j (1 + v)^2 (s)\right\|_{H^4(\mathbb R^2)} \left\|\partial_y \mathcal{H}^{k-j}v(s)\right\|_{H^4(\mathbb R^2)} \left\|\partial_y \mathcal{H}^k v(s)\right\|_{H^4(\mathbb R^2)}ds\\
    &\le B_1\sum_{j=1}^k\left\| \mathcal{H}^j (1 + v)^2 \right\|_{L^\infty([0, t]; H^4(\mathbb R^2))}\int_0^t \left\|\partial_y \mathcal{H}^{k-j}v(s)\right\|_{H^4(\mathbb R^2)} \left\|\partial_y \mathcal{H}^k v(s)\right\|_{H^4(\mathbb R^2)}ds\\
    &\le B_1\sum_{j=1}^k\left\| \mathcal{H}^j (1 + v)^2 \right\|_{L^\infty([0, t]; H^4(\mathbb R^2))}\left(\int_0^t \left\|\partial_y \mathcal{H}^{k-j}v(s)\right\|^2_{H^4(\mathbb R^2)}ds \right)^{\frac12}\left(\int_0^t \left\|\partial_y \mathcal{H}^k v(s)\right\|^2_{H^4(\mathbb R^2)}ds\right)^{\frac12}.
\end{align*}
To estimate $\left\| \mathcal{H}^j (1 + v)^2 \right\|_{L^\infty([0, t]; H^4(\mathbb R^2))}$, we use the Leibniz-type formula \eqref{Leibiniz-type formula} 
\begin{align*}
%\begin{split}
    \left\| \mathcal{H}^j (1 + v)^2 \right\|_{L^\infty([0, t]; H^4(\mathbb R^2))}&\le \sum_{i=0}^j\binom{j}{i}\left\|\mathcal{H}^iv\mathcal{H}^{j-i}v\right\|_{L^\infty([0, t]; H^4(\mathbb R^2))}\\ 
    &\lesssim \sum_{i=0}^j\binom{j}{i}\left\|\mathcal{H}^iv\right\|_{L^\infty([0, t]; H^4(\mathbb R^2))}\left\|\mathcal{H}^{j-i}v\right\|_{L^\infty([0, t]; H^4(\mathbb R^2))},
%\end{split}
\end{align*}
for $j\le k-1$, applying \eqref{hypothesis}, one has
\begin{align*}
    \left\| \mathcal{H}^j (1 + v)^2 \right\|_{L^\infty([0, t]; H^4(\mathbb R^2))}&\lesssim \sum_{i=0}^j\frac{j!}{i!(j-i)!}\cdot\frac{A^{i-1}i!}{(i+1)^2}\cdot\frac{A^{j-i-1}(j-i)!}{(j-i+1)^2}\\
    &\lesssim\frac{A^{j-2}j!}{(j+1)^2}\sum_{i=0}^j\frac{(j+1)^2}{(i+1)^2(j-i+1)^2}\lesssim\frac{A^{j-2}j!}{(j+1)^2},
\end{align*}
here, in the preceding sequence of estimates, we use
\begin{align}\label{sum}
\begin{split}
    \sum_{i=0}^j\frac{(j+1)^2}{(i+1)^2(j-i+1)^2}
    &=\sum_{0\le i\le j/2}\frac{(j+1)^2}{(i+1)^2(j-i+1)^2}+\sum_{j/2< i\le j}\frac{(j+1)^2}{(i+1)^2(j-i+1)^2}\\
    &\le\sum_{0\le i\le j/2}\frac{(j+1)^2}{(i+1)^2(j/2+1)^2}+\sum_{j/2< i\le j}\frac{(j+1)^2}{(j/2+1)^2(j-i+1)^2}\\
    &\le\sum_{0\le i\le j/2}\frac{4(j+1)^2}{(i+1)^2(j+1)^2}+\sum_{j/2< i\le j}\frac{4(j+1)^2}{(j+1)^2(j-i+1)^2}\\
    &\le8+\sum_{1\le i\le j/2}\frac{4}{i(i+1)}+\sum_{j/2< i\le j-1}\frac{4}{(j-i)(j-i+1)}\le24.
\end{split}
\end{align}
For $j=k$, applying \eqref{hypothesis}, we have 
\begin{align*}
%\begin{split}
    \left\| \mathcal{H}^k (1 + v)^2 \right\|_{L^\infty([0, t]; H^4(\mathbb R^2))}&\lesssim\left\|v\right\|_{L^\infty([0, t]; H^4(\mathbb R^2))}\left\|\mathcal{H}^{k}v\right\|_{L^\infty([0, t]; H^4(\mathbb R^2))}\\
    &\quad + \sum_{i=1}^{k-1}\binom{k}{i}\left\|\mathcal{H}^iv\right\|_{L^\infty([0, t]; H^4(\mathbb R^2))}\left\|\mathcal{H}^{k-i}v\right\|_{L^\infty([0, t]; H^4(\mathbb R^2))}\\
    &\lesssim\left\|v\right\|_{L^\infty([0, t]; H^4(\mathbb R^2))}\left\|\mathcal{H}^{k}v\right\|_{L^\infty([0, t]; H^4(\mathbb R^2))}+\frac{A^{k-2}k!}{(k+1)^2}.
%\end{split}
\end{align*}
Then, applying Proposition \ref{existence111} and \eqref{hypothesis}, we can obtain
\begin{align*}
    Q_2&\le B_3\sum_{j=1}^{k-1}\frac{A^{j-2}j!}{(j+1)^2}\cdot\frac{A^{k-j-1}(k-j)!}{(k-j+1)^2}\left(\int_0^t \left\|\partial_y \mathcal{H}^k v(s)\right\|^2_{H^4(\mathbb R^2)}ds\right)^{\frac12}\\
    &\quad+B_3\varepsilon\left\|H^{k}v\right\|_{L^\infty([0, t]; H^4(\mathbb R^2))}\left(\int_0^t \left\|\partial_y \mathcal{H}^k v(s)\right\|^2_{H^4(\mathbb R^2)}ds\right)^{\frac12}\\
    &\quad+\frac{B_3A^{k-2}k!}{(k+1)^2}\left(\int_0^t \left\|\partial_y \mathcal{H}^k v(s)\right\|^2_{H^4(\mathbb R^2)}ds\right)^{\frac12}\\
    &\le \frac{1}{24}\int_0^t \left\|\partial_y \mathcal{H}^k v(s)\right\|^2_{H^4(\mathbb R^2)}ds + 18\left(B_3\varepsilon\right)^2\left\|\mathcal{H}^{k}v\right\|^2_{L^\infty([0, t]; H^4(\mathbb R^2))}\\
    &\quad+18\left(B_3\sum_{j=1}^{k-1}\frac{A^{k-3}j!(k-j)!}{(j+1)^2(k-j+1)^2}\right)^2+18\left(B_3\frac{A^{k-2}k!}{(k+1)^2}\right)^2.
\end{align*}
We now estimate the summation over $j$ above. Using \eqref{sum} and $p!q!\le(p+q)!$ for all $p, q\in\mathbb N$, one has
\begin{align}\label{sum1}
    \sum_{j=1}^{k-1}\frac{A^{k-3}j!(k-j)!}{(j+1)^2(k-j+1)^2}&\le \frac{A^{k-3}k!}{(k+1)^2}\sum_{j=1}^{k-1}\frac{(k+1)^2}{(j+1)^2(k-j+1)^2}\le 24\frac{A^{k-3}k!}{(k+1)^2}.
\end{align}
Hence, we can get that 
\begin{align*}
    Q_2
    &\le \frac{1}{24}\int_0^t \left\|\partial_y \mathcal{H}^k v(s)\right\|^2_{H^4(\mathbb R^2)}ds + 18\left(B_3\varepsilon\right)^2\left\|\mathcal{H}^{k}v\right\|^2_{L^\infty([0, t]; H^4(\mathbb R^2))}+18\left(25B_3\frac{A^{k-2}k!}{(k+1)^2}\right)^2.
\end{align*}

It remains to consider the term $Q_3$. Applying the Cauchy-Schwarz inequality and \eqref{hypothesis}, $Q_3$ is directly estimated as
\begin{align*}
    Q_3&\le B_1\sum_{j=0}^k\left(\int_0^t\left\| \mathcal{H}^j \partial_y(1 + v)^2(s) \right\|^2_{H^4(\mathbb R^2)}ds\right)^{\frac12}\left(\int_0^t \left\|\partial_y \mathcal{H}^{k-j}v(s)\right\|^2_{H^4(\mathbb R^2)}ds\right)^{\frac{1}{2}} \left\|\mathcal{H}^k v\right\|_{L^\infty([0, t]; H^4(\mathbb R^2))}\\
    &\le B_1\sum_{j=0}^k\frac{A^{k-j-1}(k-j)!}{(k-j+1)^2}\left(\int_0^t\left\| \mathcal{H}^j \partial_y(1 + v)^2(s) \right\|^2_{H^4(\mathbb R^2)}ds\right)^{\frac12} \left\|\mathcal{H}^k v\right\|_{L^\infty([0, t]; H^4(\mathbb R^2))}
\end{align*}
We now estimate the integration above, by using the Leibniz-type formula \eqref{Leibiniz-type formula}, one has
\begin{align*}%\label{1+v-1}
%\begin{split}
    \left\| \mathcal{H}^j \partial_y(1 + v)^2 \right\|_{H^4(\mathbb R^2)}&=2\left\| \mathcal{H}^j\left( (1 + v) \partial_y(1+v)\right)\right\|_{H^4(\mathbb R^2)}\\
    &\le2\sum_{i=0}^j\binom{j}{i}\left\|\mathcal{H}^iv \partial_y\mathcal{H}^{j-i}v\right\|_{H^4(\mathbb R^2)}\\
    &\lesssim \sum_{i=0}^j\binom{j}{i}\left\|\mathcal{H}^iv\right\|_{H^4(\mathbb R^2)}\left\|\partial_y \mathcal{H}^{j-i}v\right\|_{H^4(\mathbb R^2)}.
%\end{split}
\end{align*}
Hence, we can use the Minkowski inequality, \eqref{hypothesis} and \eqref{sum} to get
\begin{align*}
    &\left(\int_0^t\left\| \mathcal{H}^j \partial_y(1 + v)^2(s) \right\|^2_{H^4(\mathbb R^2)}ds\right)^{\frac12}\\
    &\lesssim \sum_{i=0}^j\binom{j}{i}\left\|\mathcal{H}^iv\right\|_{L^\infty([0, t]; H^4(\mathbb R^2))}\left(\int_0^t\left\|\partial_y \mathcal{H}^{j-i}v(s)\right\|^2_{H^4(\mathbb R^2)}ds\right)^{\frac12}\\
    &\lesssim \sum_{i=0}^j\frac{j!}{i!(j-i)!}\cdot\frac{A^{i-1}i!}{(i+1)^2}\cdot\frac{A^{j-i-1}(j-i)!}{(j-i+1)^2}\lesssim\frac{A^{j-2}j!}{(j+1)^2}, \quad \text{for} \ j\le k-1,
\end{align*}
and using Proposition \ref{existence111}
\begin{align*}
    &\left(\int_0^t\left\| \mathcal{H}^k \partial_y(1 + v)^2(s) \right\|^2_{H^4(\mathbb R^2)}ds\right)^{\frac12}\\
    &\lesssim\varepsilon\left(\int_0^t\left\|\partial_y\mathcal{H}^{k}v(s)\right\|^2_{H^4(\mathbb R^2)}ds\right)^{\frac12}+\varepsilon\left\|\mathcal{H}^{k}v\right\|_{L^\infty([0, t]; H^4(\mathbb R^2))}+\frac{A^{k-2}k!}{(k+1)^2}.
\end{align*}
We thus use the Cauchy-Schwarz inequality and \eqref{sum1} to get
\begin{align*}
    Q_3
    &\le B_4\sum_{j=0}^{k-1}\frac{A^{k-j-1}(k-j)!}{(k-j+1)^2}\cdot\frac{A^{j-2}j!}{(j+1)^2} \left\|\mathcal{H}^k v\right\|_{L^\infty([0, t]; H^4(\mathbb R^2))}+B_4\varepsilon\left\|\mathcal{H}^k v\right\|^2_{L^\infty([0, t]; H^4(\mathbb R^2))}\\
    &\quad+B_4\varepsilon\left(\int_0^t\left\|\partial_y\mathcal{H}^{k}v(s)\right\|^2_{H^4(\mathbb R^2)}ds\right)^{\frac12}\left\|\mathcal{H}^k v\right\|_{L^\infty([0, t]; H^4(\mathbb R^2))}+B_4\frac{A^{k-2}k!}{(k+1)^2}\left\|\mathcal{H}^k v\right\|_{L^\infty([0, t]; H^4(\mathbb R^2))}\\
    &\le B_4\varepsilon\int_0^t\left\|\partial_y\mathcal{H}^{k}v(s)\right\|^2_{H^4(\mathbb R^2)}ds+\left(2B_4\varepsilon+\frac{1}{12}\right)\left\|\mathcal{H}^k v\right\|^2_{L^\infty([0, t]; H^4(\mathbb R^2))}+6\left(25B_4\frac{A^{k-2}k!}{(k+1)^2}\right)^2.
\end{align*}
As for the term $Q_4$, we have 
\begin{align*}
    Q_4 \leq 2k^2 \eta^2 (T + 1)^{2(\eta - 1)} \int_0^t \left\| \partial_y \mathcal{H}^{k - 1} v(s) \right\|^2_{H^4(\mathbb{R}^2)} ds \leq B_5 \left( \frac{A^{k - 2}k!}{(k + 1)^2} \right)^2, 
\end{align*}
here we take $B_5 = 16 \eta^2 (T + 1)^{2(\eta - 1)}$.

Consequently, combining the results of $Q_1-Q_4$, we have that for all $0<t\le T$
\begin{align}\label{integration}
\begin{split}
   &\frac12\|\mathcal{H}^k v(t)\|^2_{H^4(\mathbb R^2)}+\frac14\int_0^t\|\partial_y \mathcal{H}^k v(s)\|_{H^4(\mathbb R^2)}ds\\
   &\le B_1\int_0^t\left\|\mathcal{H}^k v(s)\right\|^2_{H^4(\mathbb R^2)}ds+\left(\frac{1}{24}+\left(B_2+B_4\right)\varepsilon\right)\int_0^t\left\|\partial_y \mathcal{H}^k v(s)\right\|^2_{H^4(\mathbb R^2)}ds\\
   &\quad+\left(18\left(B_3\varepsilon\right)^2+2B_4\varepsilon+\frac{1}{12}\right)\left\|\mathcal{H}^k v\right\|^2_{L^\infty([0, t]; H^4(\mathbb R^2))}+18\left(25\left(B_3+B_4+B_5\right)\frac{A^{k-2}k!}{(k+1)^2}\right)^2\\
   &\le B_1\int_0^t\left\|\mathcal{H}^k v(s)\right\|^2_{H^4(\mathbb R^2)}ds+\frac{1}{8}\int_0^t\left\|\partial_y \mathcal{H}^k v(s)\right\|^2_{H^4(\mathbb R^2)}ds+\frac14\left\|\mathcal{H}^k v\right\|^2_{L^\infty([0, T]; H^4(\mathbb R^2))}\\
   &\quad+18\left(25\left(B_3+B_4+B_5\right)\frac{A^{k-2}k!}{(k+1)^2}\right)^2,
\end{split}
\end{align}
if choosing $(B_2+B_4+B_5)\varepsilon\le\frac{1}{12}$ and $18\left(B_3\varepsilon\right)^2\le\frac{1}{12}$. Moving the negative integral term on the left hand side of above inequality, we deduce that
\begin{align*}
   &\frac14\left\|\mathcal{H}^k v\right\|^2_{L^\infty([0, t]; H^4(\mathbb R^2))}\le B_1\int_0^t\left\|\mathcal{H}^k v(s)\right\|^2_{H^4(\mathbb R^2)}ds+18\left(25\left(B_3+B_4+B_5\right)\frac{A^{k-2}k!}{(k+1)^2}\right)^2.
\end{align*}
Moreover, plugging it into \eqref{integration}, one has
\begin{align*}%\label{integration}
%\begin{split}
   &\|\mathcal{H}^k v(t)\|^2_{H^4(\mathbb R^2)}+\frac14\int_0^t\|\partial_y \mathcal{H}^k v(s)\|_{H^4(\mathbb R^2)}ds\\
   &\le 4B_1\int_0^t\left\|\mathcal{H}^k v(s)\right\|^2_{H^4(\mathbb R^2)}ds+72\left(25\left(B_3+B_4+B_5\right)\frac{A^{k-2}k!}{(k+1)^2}\right)^2,
%\end{split}
\end{align*}
then moving the negative integral term on the left hand side of above inequality and applying the Gronwall's inequality, it follows that 
\begin{align*}%\label{integration}
%\begin{split}
   &\|\mathcal{H}^k v(t)\|^2_{H^4(\mathbb R^2)}\le 72\left(25\left(B_3+B_4+B_5\right)\frac{A^{k-2}k!}{(k+1)^2}\right)^2 e^{4B_1T},
%\end{split}
\end{align*}
thus, we have
\begin{align*}%\label{integration}
%\begin{split}
   \|\mathcal{H}^k v(t)\|^2_{H^4(\mathbb R^2)}+\frac14\int_0^t\|\partial_y \mathcal{H}^k v(s)\|_{H^4(\mathbb R^2)}ds&\le 72\left(4B_1Te^{4B_1T}+1\right)\left(25\left(B_3+B_4+B_5\right)\frac{A^{k-2}k!}{(k+1)^2}\right)^2\\
   &\le\left(\frac{A^{k-1}k!}{(k+1)^2}\right)^2,
%\end{split}
\end{align*}
if taking $A\ge 45\left(B_3+B_4+B_5\right)\sqrt{4B_1Te^{4B_1T}+1}$.

\section{Proof of Theorem \ref{Theorem}}\label{section5}

We now give the proof of our main Theorem \ref{Theorem}: Suppose now $\eta > 2$, then $\frac \eta 2 > 1$. Let
$$
\mathcal{H}_{\eta/2} = \frac{t^{\eta/2 + 1}}{\eta/2 + 1} \partial_x + t^{\eta/2} \partial_y, 
$$
then $\partial_x$ and $\partial_y$ can be represented as the linear combination of $\mathcal{H}_{\eta/2}$ and $\mathcal{H}_{\eta}$, 
$$
\begin{cases}
t^{\eta + 1} \partial_x = -\frac{\left( \eta + 2 \right) \left( \eta + 1 \right)}{ \eta} \mathcal{H}_{\eta} +\frac{\left( \eta + 2 \right) \left( \eta + 1 \right)}{\eta} t^{\eta/2} \mathcal{H}_{\eta/2}, \\
\\
t^\eta \partial_y =  \frac{2\eta + 2}{\eta} \mathcal{H}_{\eta} - \frac{\eta + 2}{ \eta} t^{\eta/2} \mathcal{H}_{\eta/2}.
\end{cases}
$$
Then for any $\ell \in \mathbb{N}$, we compute as follows: 
\begin{align*}
& \quad t^{\left( \eta + 1 \right)\ell} \left\| \partial_x^\ell v(t) \right\|_{H^4(\mathbb{R}^2)} \\
&=\left\| \left( -\frac{\left( \eta + 2 \right) \left( \eta + 1 \right)}{ \eta} \mathcal{H}_{\eta} +\frac{\left( \eta + 2 \right) \left( \eta + 1 \right)}{\eta} t^{\eta/2} \mathcal{H}_{\eta/2} \right)^\ell v(t) \right\|_{H^4(\mathbb{R}^2)} \\
&\leq 2^\ell \left( \left\| \left( -\frac{(\eta + 2)(\eta + 1)}{\eta} \mathcal{H}_\eta \right)^\ell v(t) \right\|_{H^4(\mathbb{R}^2)} + \left\| \left( \frac{(\eta + 2)(\eta + 1)}{\eta} t^{\eta / 2} \mathcal{H}_{\eta / 2} \right)^\ell v(t) \right\|_{H^4(\mathbb{R}^2)} \right).
\end{align*}
Then by using the Proposition \ref{Proposition 3.1.}, there exists $A_1, A_2>0$ such that
$$
 \left\| \mathcal{H}^\ell_{\eta} v(t) \right\|_{H^4(\mathbb{R}^2)} \le A_1^{\ell + 1} \ell!,\ \ \ \ \  \left\| \mathcal{H}^\ell_{\eta/2} v(t) \right\|_{H^4(\mathbb{R}^2)}\le A_2^{\ell + 1} \ell!,
$$
it follows that, 
\begin{equation}\label{5-1}
    \sup_{0 < t \leq T} t^{(\eta + 1)\ell} \left\| \partial_x^\ell v(t) \right\|_{H^4(\mathbb{R}^2)} \leq (\tilde{A})^{\ell} \ell!, \quad \forall \ell \in \mathbb{N}.
\end{equation}
where $\tilde{A}$ depends on $\eta$ and $T$. 
Similarly, we can obtain 
\begin{equation}\label{5-2}
\sup_{0 < t \leq T} t^{\eta n} \left\| \partial_y^n v(t) \right\|_{H^4(\mathbb{R}^2)} \leq \left( \bar{A} \right)^n n!, \quad \forall n \in \mathbb{N}.
\end{equation}
Combining the \ref{5-1} and \ref{5-2}, by using  integration by parts, we have 
\begin{equation*}
\begin{aligned}
&\sup_{0 < t \leq T} t^{\left( \eta + 1 \right)\ell + n \eta} \left\| \partial_x^\ell \partial_y^n u(t) \right\|_{H^4(\mathbb{R}^2)} \\
&\leq \sup_{0 < t \leq T} \left( t^{2\left( \eta + 1 \right)\ell} \left\| \partial_x^{2\ell} u(t) \right\|_{H^4(\mathbb{R}^2)} \right)^{\frac{1}{2}} \left( t^{2\eta n} \left\| \partial_y^{2n} u(t) \right\|_{H^4(\mathbb{R}^2)} \right)^{\frac{1}{2}} \\
&\leq \left( (\tilde{A})^{2\ell} (2\ell)! (\bar{A})^{2n} (2n)! \right)^{\frac12} \\
&\leq L^{\ell + n + 1} \ell! n!.
\end{aligned}
\end{equation*} 
Since $v(t) = u(t) - 1$, we have then finished the proof Theorem \ref{Theorem}.

\end{proof}

\bigskip
\noindent {\bf Acknowledgements.} This work was supported by the NSFC (No.12031006) and the Fundamental Research Funds for the Central Universities of China.

\end{document}